\documentclass[a4paper,10pt]{article}
\usepackage{amsmath,amsfonts,amssymb,amsthm,enumerate,fullpage}
\usepackage{hyperref,booktabs}
\usepackage[utf8]{inputenc}

\providecommand{\CC}{\mathbb{C}}

\providecommand{\FF}{\mathbb{F}}
\providecommand{\PP}{\mathbb{P}}
\providecommand{\RR}{\mathbb{R}}

\providecommand{\ZZ}{\mathbb{Z}}

\providecommand{\cE}{\mathcal{E}}

\providecommand{\cL}{\mathcal{L}}
\providecommand{\cP}{\mathcal{P}}
\providecommand{\cR}{\mathcal{R}}
\providecommand{\cS}{\mathcal{S}}
\newcommand{\vJ}{\mathbf{J}}
\newcommand{\vI}{\mathbf{I}}

\newcommand{\witness}[2]{#1 \vDash #2}

\newtheorem{theorem}{Theorem}[section]
\newtheorem{lemma}[theorem]{Lemma}
\newtheorem{corollary}[theorem]{Corollary}
\newtheorem{conjecture}[theorem]{Conjecture}
\theoremstyle{definition}
\newtheorem{definition}[theorem]{Definition}
\newtheorem{example}[theorem]{Example}
\newtheorem{assumption}[theorem]{Assumption}

\title{Boolean degree 1 functions on some classical association schemes}

\author{Yuval Filmus\footnote{\hbadness=1500 Technion --- Israel Institute of Technology, Haifa, Israel.
This research was funded by ISF grant 1337/16.
},
Ferdinand Ihringer\footnote{\hbadness=1500 Einstein Institute of Mathematics, Hebrew University of Jerusalem, Israel.
Department of Pure Mathematics and Computer Algebra, Ghent University, Belgium. Supported by ERC advanced grant 320924.
The author is supported by a postdoctoral fellowship of the Research Foundation --- Flanders (FWO).}
}

\begin{document}

\maketitle

\begin{abstract}
  We investigate Boolean degree 1 functions for several 
  classical association schemes, including Johnson graphs, Grassmann graphs,
  graphs from polar spaces, and bilinear forms graphs, as well as some other domains such as multislices (Young subgroups of the symmetric group).
  In some settings, Boolean degree 1 functions are also known as \textit{completely regular strength 0 codes of covering radius 1},
  \textit{Cameron--Liebler line classes}, and \textit{tight sets}.
  
  We classify all Boolean degree $1$ functions on the multislice.
  On the Grassmann scheme $J_q(n, k)$ we show that all Boolean degree $1$ functions
  are trivial for $n \geq 5$, $k, n-k \geq 2$ and $q \in \{ 2, 3, 4, 5 \}$, and that, for general $q$,
  the problem can be reduced to classifying all Boolean degree $1$ functions on $J_q(n, 2)$.
  We also consider polar spaces and the bilinear forms graphs, giving evidence 
  that all Boolean degree $1$ functions are trivial for appropriate choices of the parameters.  

\end{abstract}

\section{Introduction}

Analysis of Boolean functions is a classical area of research in combinatorics and computer science, dealing with $0,1$-valued functions on finite domains. Most research in the area has focused on functions on the hypercube $H(n,2)$, a product domain which can also be realized as the tensor power $\FF_2^n$. Boolean functions on the hypercube appear naturally in various guises in theoretical computer science, combinatorics, and random graph theory, and as a consequence have been thoroughly investigated. The recent monograph of O'Donnell~\cite{RyanODonnell} provides a good exposition of the area and its applications.

In the last years researchers have been working on extending the theory to other domains, chiefly the \emph{Johnson graph} $J(n,k)$, also known as a \emph{slice} of the hypercube, which consists of all $k$-subsets of $[n] := \{1, \ldots, n\}$; consult~\cite{filmus5,filmus4,fkmw,fm,KellerKlein,doi:10.1137/100787325} for some of the work in this area.
Recently, the \emph{Grassmann graph} $J_q(n,k)$, which is the $q$-analog of the Johnson graph, has come to attention in theoretical computer science~\cite{TR16-198,TR17-094,Khot:2017:ISG:3055399.3055432,TR18-006}, but its research from the point of view of analysis of Boolean functions is at its infancy.

\smallskip

Let $x_i$ be the Boolean function on the hypercube with $x_i(S) = 1$ if $S_i = 1$ and $x_i(S) = 0$
otherwise. We say that a Boolean function has degree $d$ if we can write $f$ as a multivariate
polynomial in $x_1, \ldots, x_n$ of degree $d$. The following is folklore (see \cite[Exercise 1.19]{RyanODonnell}):
\begin{theorem}[Folklore]\label{thm:hamming_class}
  Every Boolean degree~1 function on $H(n, 2)$ is either constant or depends on a single coordinate.
\end{theorem}
Meyerowitz~\cite[Theorem 7]{Brouwer2003} extended this to $H(n,m)$, a result which we reproduce in Section~\ref{ch:groups}.
The aim of this paper is to produce such a classification of Boolean degree $1$ function
for various other structures, mostly classical association schemes.

\smallskip

Similar problems have been investigated by various researchers under different names.
In the context of distance-regular graphs (see \cite{Brouwer1989}) and $Q$-polynomial association schemes, a special class of association schemes, 
Boolean degree $1$ functions are known as \textit{completely regular strength $0$ codes of covering radius $1$} (also strength $0$ designs) \cite{Martin1994}.
Motivated by problems on permutation groups and finite geometry, Boolean degree $1$
functions are also known as \textit{tight sets} \cite{Bamberg2007,DeBruyn2010} and \emph{Cameron--Liebler line classes} \cite{Drudge1998}.
The history of Cameron--Liebler line classes is particularly complicated, as the problem was introduced
by Cameron and Liebler \cite{Cameron1982}, the term coined by Penttila \cite{Pentilla1985,Penttila1991},
and the algebraic point of view as Boolean degree $1$ functions only emerged later; 
see \cite{Vanhove2011}, in particular \S3.3.1, for a discussion of this.

Due to this variation of terminology, classification results of Boolean degree $1$ functions on $J(n, k)$ were obtained repeatedly 
in the literature (with small variations due to different definitions) at least three times, 
in \cite{Meyerowitz1992} for completely regular strength $0$ codes of covering radius $1$, in \cite{filmus5} for Boolean degree $1$ functions,
and in \cite{DeBoeck2016} for Cameron--Liebler line classes:
\begin{theorem}[Folklore]\label{thm:johnson-deg1_intro}
  Suppose that $k,n-k \geq 2$. Every Boolean degree~1 function on $J(n,k)$ is either constant or depends on a single coordinate.	
\end{theorem}
Depending on the definition used, this is either easy to observe or requires a more elaborate proof.

\smallskip

For the hypercube $H(n, 2)$, which is a product domain, and the Johnson graph $J(n, k)$ classifying Boolean degree $1$ functions is trivial,
but there are various other classical association schemes for which  classification is more difficult.
The Grassmann scheme $J_q(n, k)$ consists of all $k$-spaces of $\FF_q^n$ as vertices, two vertices being adjacent
if their meet is a subspace of dimension $k-1$. Boolean degree $1$ functions on $J_q(4, 2)$
were intensively investigated, and many non-trivial examples~\cite{Bruen1999,Cossidente2017,DeBeule2016,Feng2015,Gavrilyuk2018,Govaerts2005} 
and existence conditions~\cite{Gavrilyuk2014,Metsch2014} are known.
 
We call $1$-dimensional subspaces of $\FF_q^n$ \textit{points}, $2$-dimensional subspaces of $\FF_q^n$ \textit{lines},
and $(n-1)$-dimensional subspaces of $\FF_q^n$ \textit{hyperplanes}. For a point $p$ we define $p^+(S) = 1_{p \in S}$ and $p^-(S) = 1_{p \notin S}$,
and for a hyperplane $\pi$ we define $\pi^+(S) = 1_{S \subseteq \pi}$ and $\pi^-(S) = 1_{S \nsubseteq \pi}$.
The following was shown by Drudge for $q=3$~\cite[Theorem 6.4]{Drudge1998}; by Gavrilyuk and Mogilnykh for $q=4$~\cite[Theorem 3]{Gavrilyuk2014a};
and by Gavrilyuk and Matkin \cite{Gavrilyuk2018a,Matkin2018}
for $q=5$;
the result for $q=2$ follows easily from~\cite[Theorem 6.2]{Drudge1998}:
\begin{theorem}[Drudge, Gavrilyuk and Mogilnykh, Gavrilyuk and Matkin]\label{thm:k_eq_2_results}
  Let $q \in \{ 2, 3, 4, 5 \}$ and either (a) $n \geq 5$ or (b) $n=4$ and $q=2$. 
  Let $f$ be a Boolean degree $1$ function $f$ on $J_q(n, 2)$.
  Then $f$ or its complement $1-f$ is one of the following: $1, p^+, \pi^+, p^+ \lor \pi^+$.
  
  Here $p$ is a point and $\pi$ is a hyperplane satisfying the condition $p \notin \pi$; and $p^+ \lor \pi^+(S) = 1_{p \in S \text{ or } S \subseteq \pi}$.
\end{theorem}
For $J_q(n, k)$, some restrictions on the parameters of Boolean degree $1$ functions are known, see \cite{Metsch2017,Rodgers2018}.
Our main result for $J_q(n, k)$ is the following, which extends Theorem~\ref{thm:k_eq_2_results}:
\begin{theorem}\label{thm:grassmann_class_intro}
  Let $q \in \{ 2, 3, 4, 5 \}$, $k, n-k \geq 2$, and either (a) $n \geq 5$ or (b) $n=4$ and $q=2$. 
  Let $f$ be a Boolean degree $1$ function $f$ on $J_q(n, k)$.
  Then $f$ or its complement $1-f$ is one of the following: $1, p^+, \pi^+, p^+ \lor \pi^+$.
  
  Here $p$ is a point, $\pi$ is a hyperplane, and $p \notin \pi$.
\end{theorem}
This improves Corollary 5.5 in \cite{Rodgers2018}.
In particular, we reduce the problem to the $J_q(n, 2)$ case.
As soon as a version of Theorem \ref{thm:k_eq_2_results} (with the same classification) is shown for some prime power $q>5$ and $n \geq 5$,
Theorem \ref{thm:grassmann_class_intro} will generalize to this value of~$q$.

\smallskip

From an algebraic point of view, the Johnson graphs correspond to spherical buildings of type $A_n$~\cite[Chapter 6]{Tits1974}.
There are two more non-exceptional families of spherical buildings, $C_n$ and $D_n$, which in the finite
case contain the polar spaces $O^+(2n, q)$, $O(2n+1, q)$, $O^-(2n+2, q)$, $Sp(2n, q)$, $U(2n, q)$, and $U(2n+1, q)$.
For our purposes we refer to all of these as \emph{polar spaces}. Polar spaces arise naturally in the study
of finite classical groups, and are also relevant in other contexts, for example for quantum matroids~\cite{Terwilliger}. 
Boolean degree $1$ functions on points of polar spaces are known
as \emph{tight sets}~\cite{Bamberg2007}.
Boolean degree $1$ functions on \textit{maximals} (maximal subspaces) of polar spaces, which are also known as \textit{dual polar graphs}, were recently investigated
in~\cite{DeBoeck2017}, but the reader is warned that the definition of a Cameron--Liebler line class
in~\cite{DeBoeck2017} only corresponds to a Boolean degree $1$ function for
$O^-(2n+2, q)$, $O(2m+1, q)$, $O^+(2m, q)$, $Sp(2m, q)$, $U(2n, q)$, and $U(2n+1, q)$ (here $m=2n$).

There exist highly complicated Boolean degree $1$ functions on polar spaces, and a classification
result seems to be very hard to obtain. Our main result is as follows (see Section~\ref{ch:polar} for a more refined statement):
\begin{theorem}\label{thm:polar_class_intro}
  Let $k, n-k \geq 2$.
  Let $f$ be a Boolean degree $1$ function on the $k$-spaces of $O^+(2n, 2)$.
  Then $f$ can be written as a disjoint union of Boolean degree $1$ functions
  induced by $J_2(2n, 2)$.
\end{theorem}
Our method is more general and covers all Boolean degree $1$ functions on $k$-spaces
of polar spaces with $k, n-k \geq 2$. Our theorem is limited to the $O^+(2n, 2)$ case
as we lack a classification result similar to Theorem~ \ref{thm:k_eq_2_results} for other polar spaces.

Another important family of association schemes is the sesquilinear forms graphs. These include
the bilinear forms graphs, the alternating forms graphs, the Hermitian forms graphs, and
the symmetric bilinear forms graphs~\cite[Chapter 9.5]{Brouwer1989}.
Here we show that the family of non-trivial examples for Boolean degree $1$ functions
on $J_q(4, 2)$, $q$ odd, induces non-trivial examples on $H_q(2, 2)$.
Furthermore, we conjecture a classification of Boolean degree $1$ functions
on $H_q(\ell, k)$ for sufficiently large $k+\ell$, and verify it for $H_2(2, 2)$.

Finite permutation groups naturally give rise to Boolean degree $1$ functions. We discuss
these cases shortly in Section~\ref{ch:groups}, and then apply the classification of
Boolean degree $1$ functions on the symmetric group by Ellis, Friedgut and Pilpel~\cite{EFP} in Section~\ref{ch:multislice}
to obtain a classification of Boolean degree $1$ functions for the multislice, generalizing Theorem~\ref{thm:johnson-deg1_intro}.

We conclude with some questions for future work in Section~\ref{sec:future}.

\section{Preliminaries}

We hope that this paper is interesting for researchers in the areas of
association schemes, Boolean functions, coding theory, finite geometry, and permutation groups.
Hence, we introduce most of the relevant notation in the following.
In particular, we want to explain why the notions Boolean degree $1$ function, 
completely regular strength $0$ code, Cameron--Liebler line class, and tight set are often the same.

\subsection{Boolean functions}

In all our examples we have some form of \textit{coordinates}: 
elements of $[n]$ for the Hamming graph $H(n, 2)$, the Johnson graph $J(n, k)$, and the multislice;
\textit{points} ($1$-dimensional subspaces) of $\FF_q^n$ for the Grassmann graph $J_q(n, k)$
and for most graphs related to polar spaces; and transpositions $(i~j)$ or similar basic operations
for graphs derived from permutation groups.

We denote the constant one function by $1 = 1^+$
and the zero function by $0 = 1^-$. For a coordinate $x$, we denote the \textit{indicator
function} of $x$ by $x^+$. Similarly, $x^- = 1 - x^+$. We use the same notation
for other natural incidences such as incidence with hyperplanes.
More generally, write $f^+ = f$ and $f^- = 1-f$ for a Boolean function $f$.
For Boolean functions $f$ and $g$, we use Boolean operators such $f \vee g$, $f \wedge g$ and $f \rightarrow g$.
In this setting, a \textit{Boolean degree $1$ function} is 
a $0,1$-valued function on the vertices that can be written as $f = c + \sum_i c_i x_i$.

For a Boolean function $f$ on a domain $D$, we can identify $f$ with the set $\{ x \in D: f(x) = 1 \}$.
While we mostly use logical notation, some reader might prefer set theoretical notation.
Examples for equivalent expressions include $g \rightarrow f$ and $g \subseteq f$, $f \lor g$ and 
$f \cup g$, as well as $f \land g$ and $f \cap g$.

\subsection{Association schemes}\label{subsec:asssoc}

Delsarte established the systematic use of \textit{association schemes} as a tool in combinatorics in his PhD thesis~\cite{Delsarte1973}.
There exist several different definitions of association schemes, and we stick to what is
sometimes known as a symmetric association scheme.

\begin{definition}
  Let $X$ be a finite set, whose members are known as \emph{vertices}. A $d$-class association scheme is a pair $(X, \cR)$,
  where $\cR = \{ R_0, R_1, \ldots, R_d \}$ is a set of binary symmetric relations
  with the following properties:
 \begin{enumerate}[(a)]
  \item $\{ R_0, \ldots R_d \}$ is a partition of $X \times X$.
  \item $R_0$ is the identity relation.
  \item There are constants $p_{ij}^\ell$ such that for $x, y \in X$ with $(x, y) \in R_\ell$
  there are exactly $p_{ij}^\ell$ elements $z$ with $(x, z) \in R_i$ and $(z, y) \in R_j$.
 \end{enumerate}
\end{definition}
We denote $|X|$ by $v$. The relations $R_i$ can be described by their adjacency matrices
$A_i \in \CC^{v \times v}$ defined by
\begin{align*}
  (A_i)_{xy} = \begin{cases}
                1 & \text{ if } (x, y) \in R_i,\\
                0 & \text{ otherwise.}
               \end{cases}
\end{align*}
It is easily verified that the matrices $A_i$ are Hermitian and commute pairwise,
hence we can diagonalize them simultaneously, i.e.\ they have the same eigenfunctions.
From this we obtain pairwise orthogonal, idempotent Hermitian
matrices $E_j \in \CC^{v \times v}$ with the properties (possibly after reordering)
\begin{align}
  &\sum_{j=0}^d E_j = \vI, && E_0 = v^{-1} \vJ,\\
  &A_i = \sum_{j=0}^d P_{ji} E_j, && E_j = v^{-1} \sum_{i=0}^d Q_{ij} A_i
\end{align}
for some constants $P_{ij}$ and $Q_{ij}$. Here $\vI$ is the identity matrix, and $\vJ$ is the all-ones matrix.
A common eigenspace $V_j$ of the adjacency matrices corresponds to the row span of $E_j$; the idempotent $E_j$ is a projection onto $V_j$.
We say that a subset $\cE$ of $E_j$'s \textit{generates} $\langle E_j: 0 \leq j \leq d \rangle_\circ$ 
if all $E_j$ can be written as a (finite) polynomial in elements of $\cE$
using normal addition and the entry-wise product $\circ$.
An association scheme is \textit{$Q$-polynomial} if there exists an idempotent matrix
$E_i$ that generates $\langle E_j \rangle_\circ$. 
In this case we can rename $E_i$ to $E_1$, and uniquely order the idempotents by writing
\begin{align*}
  &E_0 = v^{-1} \vJ = c_0 E_1^{\circ 0}, && E_1 = E_1^{\circ 1},\\
  &E_2 = c_0 E_1^{\circ 0} + c_1 E_1^{\circ 1} + c_2 E_2^{\circ 2},
  &&E_j = \sum_{k=0}^j c_k E_k^{\circ k}.
\end{align*}
(In a few exceptional cases, there are two different idempotents which generate $\langle E_j \rangle_\circ$.)
In a $Q$-polynomial association scheme, a \textit{completely regular strength $0$ code of covering radius $1$}
refers to a Boolean function $f$ which is orthogonal to all $V_j$ with $j > 1$.

The Hamming graph gives rise to a $Q$-polynomial association scheme; here $(x, y) \in R_i$ if the 
Hamming distance between $x$ and $y$ is $i$ (see \cite[\S9.2]{Brouwer1989}).
It is well-known that $V_0 + V_1$ in the usual $Q$-polynomial ordering
corresponds to the span of the indicator functions $x_i^+$. Hence, $f$ being a 
completely regular strength $0$ code of covering radius $1$ just means that $f$ is a Boolean function which
can be written as $f = c + \sum_i c_i x_i^+$. Indeed, this behavior is rather typical
for many of the graphs which we investigate. Whenever this is the case, then we provide
corresponding references that indeed the span of the $x_i^+$'s corresponds
to $V_0 + V_1$.

In some other graphs under investigation, e.g.\ the symmetric group or non-maximal subspaces of polar spaces, 
which do not correspond to a relation of a $Q$-polynomial association scheme, 
we still have a similar behavior in the sense that the span of the $x_i^+$'s corresponds
to the span of very few $V_j$'s, which furthermore have a canonical description (for example, in the symmetric and general linear groups these are the isotypic components corresponding to Young diagrams with at most one cell outside the first row).

It follows from the definition of the matrix $Q$ that the function $g_{x,j}$ defined by $g_{x,j}(y) = Q_{ij}$ if $(x, y) \in R_i$, 
is orthogonal to all $V_{k}$ with $j \neq k$.
As the matrix $Q$ is easily calculated with standard techniques,
this provides an easy way of showing the non-existence of certain
Boolean degree $1$ functions with the help of an integer linear program.
For a $Q$-polynomial association scheme, this goes as follows:
\begin{enumerate}
 \item The outputs $f(y)$ of $f$ are the variables. These are $0,1$-valued, so that $f$ is Boolean.
 \item One set of constraints is $\sum_y f(y) g_{x,k}(y) = 0$ for all vertices $x$
	and all $k > 1$. This guarantees that $f$ is a degree $1$ function.
 \item Another set of constraints is $f \neq h$ for all known Boolean degree $1$ functions $h$.
\end{enumerate}
We use this technique to show some classification results
for some finite cases.

If Boolean degree $1$ functions $f$ correspond to vectors in $V_0 + V_1$,
then we can often give conditions on the size of $f$.
Let us repeat the following well-known result (see \cite[Theorem 3]{Eisfeld1998} for a proof).

\begin{lemma}\label{lem:ev_char}
  Suppose that $f$ is a Boolean function in $V_0 + V_1$.
  If $x$ is a vertex with $f(x) = 0$, then there are exactly
  $|f| \cdot (P_{01} - P_{11})/v$ vertices $y$ with $f(y) = 1$
  and $(x, y) \in R_1$.
\end{lemma}

Hence, $|f| \cdot (P_{01} - P_{11})/v$ is an integer.
For example, for $J(n, k)$ we have $P_{01} = k(n-k)$, 
$P_{11} = (k-1)(n-k-1) - 1$, and $v = \binom{n}{k}$.
Hence, $(P_{01} - P_{11})/v = k / \binom{n-1}{k-1}$.
This is a non-trivial condition: as an example, for $J(10, 4)$
this implies that the weight (number of $1$s) of a Boolean degree $1$ function is divisible by $21$.

In the following table, we list this divisibility condition for the association
schemes we study in the following sections (see the corresponding sections for the notation).
We refer to \cite[\S9]{Brouwer1989} and \cite{Brouwer2017} for the eigenvalues.

\begin{center}
\begin{tabular}{llll} \toprule
 Graph & $P_{01} - P_{11}$ & $(P_{01} - P_{11})/v$ & Conditions \\ \midrule
 $H(n, m)$ & $m$ & $m^{-n+1}$ & \\
 $J(n,k)$ & $n$ & $\frac{k}{\binom{n-1}{k-1}}$ & \\
 $J_q(n, 2)$ & $\frac{q^n-1}{q-1}$ & $\frac{q^2-1}{q^{n-1}-1}$ & $n \geq 4$ \\
 $J_q(n, k)$ & $\frac{q^n-1}{q-1}$ & $\prod_{i=2}^k \frac{q^i-1}{q^{n-i+1}-1}$ & $n \geq 4$, $k, n-k \geq 2$ \\
 $H_q(\ell, k)$ & $q^{k+1}$ & $q^{1-k(\ell-1)}$ & $n \geq 4$, $k \geq \ell \geq 2$ \\
 $C_q(n, n, e)$ & $q^{n-1+e}+1$ & $\prod_{i=1}^{n-1} (q^{i-1+e}+1)^{-1}$ & \\
 $Q_q(n)$ & $q^{2n-1}$ & $q^{-(n-1)^2}$ &  \\
 $A_q(n)$ & $q^{2n-3}$ & $q^{-n(n-1)/2 + 2n - 3}$ &  \\
 \bottomrule
\end{tabular}
\end{center}
Note that in specific cases more detailed conditions are known, see for example
\cite{Gavrilyuk2014} for $J_q(4, 2)$.

\paragraph{Finite geometry and permutation groups}

The connection between association schemes, finite geometry and permutation groups
is well-explained in \cite{Vanhove2011}, in particular \S3.3.2.

\subsection{Coordinate-induced subgraphs}\label{sec:coord}

We say that an induced (coordinatized) subgraph $\Gamma'$ of a (coordinatized) graph 
$\Gamma$ is a \textit{coordinate-induced} subgraph if the indicator functions of coordinates $x_i$ of $\Gamma'$ 
are the restrictions of the indicator functions of coordinates $x_i$ of $\Gamma$.
For example, the Johnson graph $J(n, k)$ is a coordinate-induced subgraph of $H(n, 2)$, in which we only consider
vertices of $H(n, 2)$ with exactly $k$ entries equal to $1$.

In contrast, the Johnson graph $J(3, 2)$ can be naturally embedded into the Fano plane $J_2(3, 2)$,
e.g.\ if $\{ 1, 2, 3 \}, \{ 1, 4, 5 \}, \{ 1, 6, 7 \}, \{ 2, 4, 7 \}, \{ 2, 5, 6 \}, 
\{ 3, 4, 6 \}, \{ 3, 5, 7 \}$ is our model of the Fano plane, then the induced subgraph on 
$\{ 1, 2, 3 \}, \{ 1, 4, 5 \}, \{ 3, 5, 7 \}$ is $J(3, 2)$. But this embedding
is not coordinate-induced as the degree~$1$ polynomials on $J(3, 2)$ correspond to the
degree~$1$ polynomials in $x_1,x_3,x_5$ rather than all of $x_1, \ldots, x_7$.

We make the following easy observation.

\begin{lemma}\label{lem:natural_embedding_rule}
  Let $\Gamma$ be a (coordinatized) graph and let $\Gamma'$ be a coordinate-induced
  subgraph of $\Gamma$. If $f$ is a degree $d$ function of $\Gamma$, then 
  the restriction of $f$ to $\Gamma'$ is a degree $d$ function.
\end{lemma}

Another example for a coordinate-induced subgraph is the natural embedding of $J_2(4, 2)$ in $J(15, 3)$.
Clearly, this is a coordinate-induced subgraph, but we already
saw that the degree $1$ functions on $J(15, 3)$ are either constant or $x^\pm$,
whereas we have more examples of degree~$1$ functions on $J_2(4, 2)$.
Hence, it is clear that the reverse statement is not true.

\section{Johnson graphs}

As a warm-up for the Grassmann scheme, we will provide a new proof for Theorem \ref{thm:johnson-deg1_intro}.
Our proof for the analogous result for the Grassmann scheme will be similar though much longer
and involving many more cases. This is essentially due to the fact that on the Johnson scheme
$x \notin S$ and $x \in \overline{S}$ for a subset $S$ are the same, whereas on the Grassmann scheme 
$x \notin S$ and $x \in S^\perp$ for a subspace $S$ are different (where $S^\perp$ denote the orthogonal complement of $S$).

Let $P(k,\ell)$ be the following statement:

\begin{center}
\emph{Every Boolean degree~1 function on $J(k+\ell,k)$ is equal to $0^\pm$ or to $x^\pm$ for some $x \in [n]$.}
\end{center}

The proof is by induction. First we show $P(2, 2)$, then we show that $P(k, \ell)$ implies $P(k+1, \ell)$ and $P(k, \ell+1)$.

\textbf{Base case.} There are $2^{2^4}$ Boolean functions on $J(4, 2)$. These can be checked exhaustively and then $P(2, 2)$ follows.

\textbf{Inductive step.}
If $P(k, \ell)$ is true, then $P(\ell, k)$ is true as $J(k+\ell, k)$ and $J(k+\ell, \ell)$ are isomorphic.
Hence, if we show that $P(k, \ell)$ implies $P(k+1, \ell)$, then we also show that $P(k, \ell)$ implies $P(k, \ell+1)$.
Therefore it suffices to show that $P(k,\ell)$ implies $P(k+1,\ell)$. Let $f$ be a Boolean degree~1 function on $J(k+\ell+1,k+1)$. The idea is to consider \emph{restrictions} of $f$ into subdomains isomorphic to $J(k+\ell,k)$. For every $a \in [k+\ell+1]$, let $f_a$ be the restriction of $f$ to the sets containing $a$. Since the domain of $f_a$ is isomorphic to $J(k+\ell,k)$, we know that $f_a \in \{0,1,x^\pm\}$, where $x \neq a$.

The main idea of the proof is to consider the possible values of $f_a,f_b$ for $a \neq b$.

\begin{lemma} \label{lem:johnson-deg1-pair}
 Let $f\colon J(k+\ell+1,k+1) \to \{0,1\}$, where $k,\ell \geq 2$, and suppose that $P(k,\ell)$ holds.
 
 For any $a \neq b$, one of the following options holds:
\begin{enumerate}
 \item $f_a = f_b = 1^\pm$.
 \item $f_a = f_b = x^\pm$, where $x \neq a,b$.
 \item $f_a = 1^\pm$ and $f_b = a^\pm$ (with the same sign), or vice versa.
\end{enumerate}
\end{lemma}
\begin{proof}
 We illustrate our method by showing that it cannot be the case that $f_a = x^-$ and $f_b = y^-$, where $a,b,x,y$ are all distinct. Since $k+1 \geq 3$ and $\ell \geq 2$, there exists a set $S$ which contains $a,b,x$ but not $y$. We will denote this for brevity $\witness{S}{abx\bar{y}}$ (in the sequel we will sometimes use the notations $x^+ = x$ and $x^- = \bar{x}$ in this context). Then $f_a(S) = 0$ (since $x \in S$) whereas $f_b(S) = 1$ (since $y \notin S$).
More generally, a set conforming to such a condition exists as long as we specify at most three elements in the set and at most two out of the set. 
  
 We consider three cases: $f_a = 1^\pm$, $f_a = b^\pm$, and $f_a = x^\pm$ for $x \neq b$. Below we use $x,y$ to denote two different indices which differ from $a,b$.
 \begin{itemize}
  \item \textbf{Case 1:} $f_a = 1^+$. In this case, we claim that $f_b \in \{1^+,a^+\}$. We show this by ruling out all other cases:
  \begin{itemize}
   \item $f_b \in \{1^-,a^-\}$: Let $\witness{S}{ab}$. Then $f_a(S) = 1$ but $f_b(S) = 0$.
   \item $f_b = x^\pm$: Let $\witness{S}{abx^\mp}$. Then $f_a(S) = 1$ but $f_b(S) = 0$.
  \end{itemize}
  \item \textbf{Case 2:} $f_a = b^+$. In this case, we claim that $f_b \in \{1^+,a^+\}$. We show this by ruling out all other cases:
  \begin{itemize}
   \item $f_b \in \{1^-,a^-\}$: Let $\witness{S}{ab}$. Then $f_a(S) = 1$ but $f_b(S) = 0$.
   \item $f_b = x^\pm$: Let $\witness{S}{abx^\mp}$. Then $f_a(S) = 1$ but $f_b(S) = 0$.
  \end{itemize}
  \item \textbf{Case 3:} $f_a = x^+$. In this case, we claim that $f_b = x^+$ as well. We show this by ruling out all other possibilities:
  \begin{itemize}
   \item $f_b \in \{1^+,a^+\}$: Let $\witness{S}{ab\bar{x}}$. Then $f_a(S) = 0$ but $f_b(S) = 1$.
   \item $f_b \in \{1^-,a^-,x^-\}$: Let $\witness{S}{abx}$. Then $f_a(S) = 1$ but $f_b(S) = 0$.
   \item $f_b = y^\pm$: Let $\witness{S}{ab\bar{x}y^\pm}$. Then $f_a(S) = 0$ but $f_b(S) = 1$.
  \end{itemize}
 \end{itemize}
 
 If $f_a \in \{1^-,b^-,x^-\}$ then we consider $1-f$ (and so $1-f_a,1-f_b$) to obtain analogous results.
 
 This case analysis shows that the possible values of $f_a,f_b$ are
\begin{align}
 (f_a,f_b) \in \{(1^\pm,1^\pm), (1^\pm,a^\pm), (b^\pm, 1^\pm), (b^\pm, a^\pm), (x^\pm,x^\pm)\},\label{eq:step_between_johnson_sodoku}
\end{align}
 where in all cases the signs agree. It remains to rule out the case $(f_a,f_b) = (b^\pm,a^\pm)$. To this end, we pick a third coordinate $c \neq a,b$, and consider the possible value of $f_c$ when $f_a=b^+$ and $f_b=a^+$. Considering the pair $f_a,f_c$, we see that $f_c = b^+$. Considering the pair $f_b,f_c$, we see that $f_c = a^+$. We reach a contradiction by considering $\witness{S}{a\bar{b}c}$, since $b^+(S) = 0$ whereas $a^+(S) = 1$.
\end{proof}

From here the proof is very easy.

\begin{lemma} \label{lem:johnson-deg1-inductive}
 Let $f\colon J(k+\ell+1,k+1) \to \{0,1\}$, where $k,\ell \geq 2$, and suppose that $P(k,\ell)$ holds. Then $f \in \{1^\pm,x^\pm\}$.
 
 In other words, $P(k,\ell)$ implies $P(k+1,\ell)$. Similarly, $P(k,\ell)$ implies $P(k,\ell+1)$.
\end{lemma}
\begin{proof}
 We will find the possible values of $\{f_i : i \in [k+\ell+1]\}$. Since the domains of $f_i$ cover all of $J(k+\ell+1,k+1)$, this will allow us to determine~$f$.

 If $f_a = x^\pm$ for some $a$ then Lemma~\ref{lem:johnson-deg1-pair} shows that $f_b = x^\pm$ for all $b \neq x$, and $f_x = 1^\pm$. It is not hard to check that $f = x^+$.
 
 Otherwise, $f_a = 1^\pm$ for all $a$. Lemma~\ref{lem:johnson-deg1-pair} shows that all $f_a$ are equal to the same constant, and so $f$ itself is equal to this constant as well.
\end{proof}

\section{Grassmann graphs}

We denote the Grassmann graph by $J_q(n,k)$.
It is well-known that completely regular strength 0 codes of covering radius 1 correspond
to Boolean degree $1$ functions, as the first eigenspace of the scheme
is spanned by the $x_i$ \cite[\S3.2]{Vanhove2011}.

Let $f$ be a Boolean degree 1 function on $J_q(n, k)$.
We have the following trivial examples for $f$:
\begin{enumerate}[(a)]
 \item $f = 1^\pm$.
 \item $f = p^\pm$ for a point $p$.
 \item $f = \pi^\pm$ for a hyperplane $\pi$.
 \item $f = (p \lor \pi)^\pm$ for a point $p$ and a hyperplane $\pi$,
  where $p \notin \pi$.
\end{enumerate}
We call $f$ \textit{trivial} if $f$ is one of the above. 
Let us restate Theorem \ref{thm:grassmann_class_intro} slightly differently:
\begin{theorem}\label{thm:grassmann_class}
  Let $q \in \{ 2, 3, 4, 5 \}$. Then all Boolean degree $1$ functions on $J_q(n, k)$
  are trivial if $k, n-k \geq 2$ and either (a) $n \geq 5$ or (b) $n=4$ and $q=2$.
\end{theorem}

The rest of this section is concerned with a proof of this result.

Our proof for the Grassmann scheme has a similar structure as our proof for the Johnson scheme.
One central difference between $J(n, k)$ and $J_q(n, 2)$ is that for $J(4, 2)$ all Boolean degree $1$
functions are trivial, while there do exist non-trivial examples for $J_q(4, 2)$ when $q > 2$.
For the base case we use Theorem \ref{thm:k_eq_2_results}.

As mentioned in the introduction, 
there are no non-trivial examples for Boolean degree $1$ functions on $J_q(5, 2)$ known if $q>2$, so the following conjecture is (in some sense) the strongest possible.
\begin{conjecture}\label{conj:k_eq_2}
  A Boolean degree $1$ function on $J_q(n, 2)$ is trivial if either (a) $n \geq 5$ or (b) $n=4$ and $q=2$.
\end{conjecture}
Our proof of Theorem \ref{thm:grassmann_class} implies that if Conjecture $\ref{conj:k_eq_2}$
is true, then also the following holds.
\begin{conjecture}\label{conj:k_gen}
  A Boolean degree $1$ function on $J_q(n, k)$ is trivial if $k, n-k \geq 2$ and either (a) $n \geq 5$ or (b) $n=4$ and $q=2$.
\end{conjecture}
We believe at least the following to be true.
\begin{conjecture}\label{conj:k_gen_weak}
  Let $q$ be a prime power. 
  Then there exists a constant $n_{q}$ such that
  a Boolean degree $1$ function on $J_q(n, k)$ is trivial for all $n \geq n_q$ if $k, n-k \geq 2$.
\end{conjecture}

Lemma \ref{lem:char_by_quotient_case_pt}, Lemma \ref{lem:char_by_quotient_case_pl},
and Lemma \ref{lem:char_by_quotient_case_pt_pl} are the analog of Lemma \ref{lem:johnson-deg1-pair}
up to Equation \eqref{eq:step_between_johnson_sodoku}.

\begin{lemma}\label{lem:char_by_quotient_case_pt}
  Let $n \geq 2k \geq 4$ and let $f$ be a trivial Boolean degree $1$ function on $J_q(n, k)$.
  Fix a line $\ell$ and a point $a \in \ell$. Suppose that for all $k$-spaces $K$ through $a$ the following holds:
  \begin{enumerate}[(a)]
   \item $f(K) = 1$ if $\ell \subseteq K$,
   \item $f(K) = 0$ otherwise.
  \end{enumerate}
  Then one of the following cases occurs:
  \begin{enumerate}
    \item $f = p^+$ for some point $p \in \ell \setminus \{ a \}$.
    \item $f = p^+ \lor \pi^+$ for some point $p \in \ell \setminus \{ a \}$ and some hyperplane $\pi \ni a$.
  \end{enumerate}
\end{lemma}
\begin{proof}
  In the following, $\pi$ is always a hyperplane.
  If $p$ and $\pi$ occur in the same argument, then $p \notin \pi$.
  We denote $\mathbb{F}_q^n$ by $V$.
  We will reuse the symbols $K$ and $K'$ in every paragraph.
  
  Clearly, $f \notin \{ 0, 1 \}$.
  
  Suppose for a contradiction that $f=\pi^+$. Let $K$ be a $k$-space that contains $\ell$.
  Then $f(K) = 1$. Hence, $\ell \subseteq \pi$. First we handle the case that $k\geq 3$. 
  Let $p$ be a point not in $\pi$ and let $L$ be a $(k-1)$-space in $\pi$ which contains $\ell$.
  Then $K' := \langle p, L \rangle$ is a $k$-space which contains $\ell$. Hence,
  $f(K') = 1$. As $f=\pi^+$, this contradicts $K' \nsubseteq \pi$. Now we handle the case that $k=2$.
  We have $\dim(\pi) = n-1 \geq 3$. Hence, there exists a point $r$ in $\pi \setminus \ell$.
  Then $f(\langle a, r \rangle) = 0$. As $f=\pi^+$, this contradicts $\langle a, r \rangle \subseteq \pi$.
  
  Suppose that $f = p^-$. If $k > 2$, then any $k$-space $K$ through $\langle \ell, p \rangle$ satisfies
  $f(K) = 1$. This contradicts $f = p^-$. If $k=2$, then, as $n \geq k+2$, there exists a $k$-space $K'$ through
  $a$ which does not contain $p$ or $\ell$. Hence, $f(K') = 0$. 
  As $f=p^-$, this contradicts $p \notin K$.

  Suppose that $f = \pi^-$. All $k$-spaces $K$ through $\ell$ satisfy $f(K) = 1$, hence $\ell \nsubseteq \pi$.
  There exist $k$-spaces $K$ on $a$ with $f(K) = 0$, hence $a \in \pi$. Let $L$ be a $(k-1)$-space through $a$ in $\pi$.
  We have $\dim(V/L) = n-k+1 \geq 3$, $\dim(\pi/L) = n-k$, and $\dim(\langle \ell, L \rangle/L) = 1$.
  Hence, there exists a point $r$ such that $\langle r, L \rangle/L \notin \pi/L$
  and $\langle r, L \rangle/L \neq \langle \ell, L \rangle/L$.
  Hence, $K := \langle r, L \rangle$ does not contain $\ell$, so $f(K) = 0$. 
  As $f = \pi^-$, but $K \nsubseteq \pi$, this is a contradiction.
  
  Suppose that $f = p^- \land \pi^-$. All $k$-spaces $K$ through $\ell$ satisfy $f(K) = 1$, hence $\ell \nsubseteq \pi$.
  If $a \notin \pi$, then we can proceed as in the case $f = p^-$.
  If $a \in \pi$, then let $L$ be a $(k-1)$-space through $a$ in $\pi$.
  As $n \geq k+2$, $\dim(V/L) = n-k+1 \geq 3$. As $p \notin \pi \supseteq L$, $\dim(\langle p, L \rangle/L) = 1$.
  Furthermore, $\dim(\pi/L) = n-k$. Hence, there exists a point $r$ such that $\langle r, L \rangle/L \notin \pi/L$,
  and $\langle r, L \rangle/L \neq \langle p, L \rangle/L, \langle \ell, L \rangle/L$.
  Then $K := \langle r, L \rangle$ is a $k$-space which does not contain $\ell$, so $f(K) = 0$.
  As $f = p^- \land \pi^-$, this contradicts $p \notin K \nsubseteq \pi$.
  
  Suppose that $f=p^+ \lor \pi^+$ with $a \in \pi$. Let $L$ be a $(k-1)$-space through $a$
  in $\pi$. As $\dim(\pi/L) = n-k \geq 2$, there exists a point $r$ such that 
  $\langle r, L \rangle/L \in \pi/L$ and $\langle r, L \rangle/L \neq \langle \ell, L \rangle/L$.
  Then $K := \langle r, L \rangle$ is a $k$-space which does not contain $\ell$, so $f(K) = 0$.
  As $f = p^+ \lor \pi^+$, this contradicts $K \subseteq \pi$. Hence, $a \notin \pi$.
  
  Hence, $f = p^+$ or $f = p^+ \lor \pi^+$ with $a \notin \pi$ are the only cases left.
  We have to show that $p \in \ell \setminus \{ a \}$. As there exist $k$-spaces $K$ and $K'$
  through $a$ with $f(K) = 1 \neq 0 = f(K')$, $p \neq a$. Suppose that $p \notin \ell$. Then there exists a
  $k$-space $K$ through $\langle a, p \rangle$ which does not contain $\ell$, so $f(K) = 0$.
  As $f = p^+$ or $f = p^+ \lor \pi^+$, this contradicts $p \in K$. Hence, $p \in \ell$.
\end{proof}

\begin{lemma}\label{lem:char_by_quotient_case_pl}
  Let $n \geq 2k \geq 4$ and let $f$ be a trivial Boolean degree $1$ function on $J_q(n, k)$.
  Fix a hyperplane $\pi$ and a point $a \in \pi$. Suppose that for all $k$-spaces $K$ through $a$ the following holds:
  \begin{enumerate}[(a)]
   \item $f(K) = 1$ if $K \subseteq \pi$,
   \item $f(K) = 0$ otherwise.
  \end{enumerate}
  Then $f = \pi^+$.
\end{lemma}
\begin{proof}
  In the following, $p$ is always some point, and $\tilde{\pi}$ is some hyperplane.
  If $p$ and $\tilde{\pi}$ occur in the same argument, then $p \notin \tilde{\pi}$.
  We denote $\mathbb{F}_q^n$ by $V$.
  We will reuse the symbols $K$ and $K'$ in every paragraph.
  
  Clearly, $f \notin \{ 0, 1 \}$.
  
  Suppose that $f= p^+$ or $f = p^+ \lor \pi^+$, where $p \notin \pi$ in the latter case.
  Let $K$ be a $k$-space through $a$ and $p$. Since $p \in K$ and $f = p^+$ or $f = p^+ \lor \pi^+$, we have $f(K) = 1$, and so
  $K \subseteq \pi$. Hence, $p \in \pi$, ruling out the case $f = p^+ \lor \pi^+$.
  As $\dim(\pi) = n-1 > k$, there exists a $k$-space $K'$ 
  in $\pi$ through $a$ disjoint to $p$. As $K \subseteq \pi$,
  $f(K) = 1$. As $f=p^+$, this contradicts $p \notin K$.
  
  Suppose that $f=p^-$. If $p \in \pi$, then there exists a $k$-space
  $K$ through $a$ and $p$ in $\pi$. Hence, $f(K) = 1$, which contradicts $f=p^-$.
  So $p \notin \pi$. Let $L$ be a $(k-1)$-space through $a$ in $\pi$.
  Then $\dim(V/L) = n-k+1$, $\dim(\pi/L) = n-k$, and $\dim(\langle p, L \rangle/L)=1$.
  Hence, there exists a point $r$ such that $\langle r, L \rangle/L \notin \pi/L$
  and $\langle r, L \rangle/L \neq \langle p, L \rangle/L$. Let $K := \langle r, L \rangle$.
  As $K \nsubseteq \pi$, $f(K) = 0$. As $f=p^-$, this contradicts $p \notin K$.
  
  Suppose that $f = \tilde{\pi}^-$ for some hyperplane $\tilde{\pi}$
  or $f = p^- \land \tilde{\pi}^-$ for some hyperplane $\tilde{\pi}$
  with $p \notin \tilde{\pi}$.
  If $a \in \tilde{\pi}$, then $a \in \pi \cap \tilde{\pi}$, where 
  $\dim(\pi \cap \tilde{\pi}) \geq n-2 \geq k$. Hence, there exists a $k$-space $K$
  through $a$ in $\pi \cap \tilde{\pi}$. As $K \subseteq \pi$, then $f(K) = 1$.
  As $f = \tilde{\pi}^-$ or $f = p^- \land \tilde{\pi}^-$, this contradicts $K \subseteq \tilde{\pi}$.
  If $a \notin \tilde{\pi}$ and $f = \tilde{\pi}^-$, let $L$ be a $(k-1)$-space through $a$ in $\pi$.
  Let $r$ be a point not in $\pi$. Then $K' := \langle r, L \rangle$ is not in $\pi$.
  Hence, $f(K') = 0$. As $a \notin \tilde{\pi}$, $K' \nsubseteq \tilde{\pi}$. 
  This contradicts $f = \tilde{\pi}^-$.
  If $a \notin \tilde{\pi}$ and $f = p^- \land \tilde{\pi}^-$, then there exists,
  as $\dim(\pi) = n-1 > k$, a $(k-1)$-space $L$ through $a$ in $\pi$
  which does not contain $p$. As $\dim(V/L) = n-k+1$, $\dim(\pi/L) = n-k$,
  and $\dim(\langle p, L \rangle/L) = 1$, there exists a point $r$
  such that $\langle r, L \rangle/L \notin \pi/L$ and 
  $\langle r, L \rangle/L \neq \langle p, L \rangle/L$. Then $K' := \langle r, L \rangle$
  is not in $\pi$, so $f(K) = 0$. As $f = p^- \land \tilde{\pi}^-$,
  this contradicts $p \notin K$ and $K \nsubseteq \pi$.
  
  Suppose that $f = \tilde{\pi}^+$ with $\tilde{\pi} \neq \pi$.
  Let $K$ be a $k$-space in $\pi$ through $a$ with $K \nsubseteq \tilde{\pi}$.
  Then $f(K) = 1$. As $f = \tilde{\pi}^+$, this contradicts $K \nsubseteq \tilde{\pi}$.
  
  Suppose that $f = p^+ \lor \tilde{\pi}^+$ for some $p \notin \tilde{\pi}$.
  If $\pi \neq \tilde{\pi}$, let $L$ be a $(k-1)$-space through $a$ in $\pi \cap \tilde{\pi}$.
  As $\dim(\pi/L) = n-k\geq 2$ and $\dim((\pi \cap \tilde{\pi})/L) = n-k-1$, 
  there exists a point $r \in \pi \setminus \tilde{\pi}$ such that 
  $\langle r, L \rangle/L \neq \langle p, L \rangle/L$. 
  Hence, $K := \langle r, L \rangle \subseteq \pi$, so $f(K) = 1$.
  As $f(K) = p^+ \lor \tilde{\pi}^+$, this contradicts $p \notin K$ and
  $K \nsubseteq \tilde{\pi}$.
  
  Hence, $f = \pi^+$ is the only option left.
\end{proof}

\begin{lemma}\label{lem:char_by_quotient_case_pt_pl}
  Let $n \geq 2k \geq 4$ and let $f$ be a trivial Boolean degree $1$ function on $J_q(n, k)$.
  Fix a hyperplane $\pi$, a point $a \in \pi$, and a line $\ell \nsubseteq \pi$ through $a$. 
  Suppose that for all $k$-spaces $K$ through $a$ the following holds:
  \begin{enumerate}[(a)]
   \item $f(K) = 1$ if $K \subseteq \pi$,
   \item $f(K) = 1$ if $\ell \subseteq K$,
   \item $f(K) = 0$ otherwise.
  \end{enumerate}
  Then $f = p^+ \lor \pi^+$ for some point $p \in \ell \setminus \{ a \}$.
\end{lemma}
\begin{proof}
  In the following, $p$ is always some point, and $\tilde{\pi}$ is some hyperplane.
  If $p$ and $\tilde{\pi}$ occur in the same argument, then $p \notin \tilde{\pi}$.
  We denote $\mathbb{F}_q^n$ by $V$.
  We will reuse the symbols $K$ and $K'$ in every paragraph.
  
  Clearly, $f \notin \{ 0, 1 \}$.
  
  Suppose that $f = p^-$. 
  If $k > 2$ or $p \in \ell$, then there exists a $k$-space $K$ through $\langle \ell, p \rangle$.
  Hence, $f(K) = 1$. As $f = p^-$, this contradicts $p \in K$.
  Hence, $k=2$ and $p \notin \ell$. Clearly, there exists a $2$-space $K$ through $a$
  with $K \nsubseteq \pi$ and $K \neq \ell$ (and so $p \notin K$). Then $f(K) = 0$. As $f=p^-$, this contradicts $p \notin K$.
  
  Suppose that $f = \tilde{\pi}^-$.
  All $k$-spaces $K$ through $\ell$ satisfy $f(K)=1$, hence $\ell \nsubseteq \tilde{\pi}$.
  There exist $k$-spaces $K'$ through $a$ with $f(K') = 0$, hence, as $f = \tilde{\pi}^-$, $a \in \tilde{\pi}$.
  As $\dim(\pi \cap \tilde{\pi}) \geq n-2 \geq k$, there exists a $(k-1)$-space $L$ through $a$
  in $\pi \cap \tilde{\pi}$. We have $\dim(V/L) = n-k+1 \geq 3$, $\dim(\pi/L) = \dim(\tilde{\pi}/L) = n-k$,
  and $\dim(\langle \ell, L \rangle/L) = 1$. Hence, there exists a point $r$ such that
  $\langle r, L \rangle/L \nsubseteq \pi/L, \tilde{\pi}/L$ and $\langle r, L \rangle/L \neq \langle p, L \rangle/L$.
  Hence, $K := \langle r, L \rangle$ satisfies $f(K) = 0$. As $f = \tilde{\pi}^-$,
  this contradicts $K \nsubseteq \tilde{\pi}$.
  
  Suppose that $f = p^- \land \tilde{\pi}^-$ for some hyperplane $\tilde{\pi}$.
  All $k$-spaces $K$ through $\ell$ satisfy $f(K) = 1$, hence $\ell \nsubseteq \tilde{\pi}$.
  If $a \notin \tilde{\pi}$, then we can proceed as in the case $f = p^-$.
  If $a \in \tilde{\pi}$, then, as $\dim(\pi \cap \tilde{\pi}) \geq n-2 \geq k$, then there exists
  a $(k-1)$-space $L$ through $a$ in $\pi \cap \tilde{\pi}$.
  As $\dim(V/L) = n-k+1 \geq 3$, $\dim(\pi/L) = \dim(\tilde{\pi}/L) = n-k$, and $\dim(\langle \ell, L \rangle/L) = \dim(\langle p, L \rangle/L) = 1$,
  there exists a point $r$ such that $\langle r, L \rangle/L \nsubseteq \pi/L, \tilde{\pi}/L$ and 
  $\langle r, L \rangle/L \neq \langle \ell, L \rangle/L, \langle p, L \rangle/L$. Hence, $K := \langle r, L \rangle$ does
  not contain $\ell$ and does not lie in $\pi$. Hence, $f(K) = 0$. As $f=p^- \land \tilde{\pi}^-$,
  this contradicts that $p \notin K \nsubseteq \tilde{\pi}$.
  
  Suppose that $f = p^+$. If $p = a$, then we can find a $k$-space $K$ through $a$ which doesn't contain $\ell$ and is not in $\pi$.
  Then $f(K) = 0$, contradicting $f = p^+$. Hence, $p \neq a$.
  As $\dim(\pi)=n-1 > k \geq 2$, there exists a $k$-space $K$
  in $\pi$ through $a$ not containing $p$. As $K \subseteq \pi$, $f(K)=1$. As $f=p^+$,
  this contradicts $p \notin K$.
  
  Suppose that $f = \tilde{\pi}^+$. As not all $k$-spaces through $\ell$ lie in $\tilde{\pi}$, there 
  exists a $k$-space $K$ that contains $\ell$ and not in $\tilde{\pi}$. Then $f(K) = 1$.
  As $f = \tilde{\pi}^+$, this contradicts $\ell \nsubseteq K$.
  
  Suppose that $f = p^+ \lor \tilde{\pi}^+$ for some $p \notin \ell$.
  Let $L$ be a $(k-1)$-space through $\langle a, p \rangle$ with $\ell \nsubseteq L$.
  As $\dim(V/L) \geq n-k+1 \geq 3$, $\dim(\pi/L) = n-k$ and $\dim(\langle p, L\rangle/L) = 1$,
  there exists a point $r$ such that $\langle r, L \rangle/L \nsubseteq \pi/L$
  and $\langle r, L \rangle/L \neq \langle p, L \rangle/L$. Then $K := \langle r, L \rangle$
  does not contain $\ell$ and does not lie in $\pi$, so $f(K) = 0$.
  As $f = p^+ \lor \tilde{\pi}^+$, this contradicts $p \in K$.
  
  Suppose that $f = p^+ \lor \tilde{\pi}^+$ for $p = a$. 
  Then $a \notin \tilde{\pi}$, so we can proceed as in the case $f = p^+$ and obtain a contradiction.
  
  Suppose that $f = p^+ \lor \tilde{\pi}^+$ for $p \in \ell \setminus \{ a \}$ and $\pi \neq \tilde{\pi}$.
  If $a \notin \tilde{\pi}$ then we can proceed as in the case $f = p^+$ and obtain a contradiction. Hence, $a \in \tilde{\pi}$.
  Let $L$ be a $(k-1)$-space in $\pi \cap \tilde{\pi}$ which contains $a$. 
  As $\dim(\pi/L) = n-k \geq 2$, $\dim((\pi \cap \tilde{\pi})/L) = n-k-1$, and $\dim(\langle p, L \rangle/L) = 1$, there exists a point
  $r \in \pi \setminus \tilde{\pi}$ such that $\langle r, L \rangle/L \neq \langle p, L \rangle/L$.
  Hence, $K := \langle r, L \rangle \subseteq \pi$, so $f(K) = 1$. As $f = p^+ \lor \tilde{\pi}^+$,
  this contradicts $p \notin K$ and $K \nsubseteq \tilde{\pi}$.
  
  The only option left is that $f = p^+ \lor \pi^+$ with $p \in \ell \setminus \{ a \}$.
\end{proof}

The next results up to Lemma \ref{lem:hyp_compl} complete the analog of Lemma \ref{lem:johnson-deg1-pair}.
We write $f_a$ for the restriction of $f$ to all subspaces containing $a$, a domain isomorphic to $J_q(n-1,k-1)$.

From now on we assume the following until the end of this section:
\begin{assumption}
 \begin{enumerate}[(a)]
  \item We have $2 < k < n-2$ (and so $n \geq 6$).
  \item All Boolean degree $1$ functions on $J_q(n-1, k-1)$ are trivial.
 \end{enumerate}
\end{assumption}
We want to remind the reader that restrictions of a Boolean degree $1$ function 
$f$ to a subspace of $\FF_q^n$ or to a quotient space of $\FF_q^n$ are still
Boolean degree $1$ functions (as follows easily from
the discussion in Section~\ref{sec:coord}).

\begin{corollary}\label{cor:antiflag_hull}
  Let $f_a = p^+ \lor \pi^+$ (if $a \notin \pi$, this is the same as $f_a = p^+$).
  Then the following holds:
  \begin{enumerate}[(a)]
   \item If $b \notin \langle a, p \rangle$ and $a,b \in \pi$, 
    then $f_b = \tilde{p}^+ \lor \pi^+$ for some $\tilde{p} \in \langle a, p \rangle$.
    \item If $b \notin \langle a, p \rangle$ and $b \notin \pi$, 
    then $f_b = \tilde{p}^+$ ($= \tilde{p}^+ \lor \pi^+$ as $b \notin \tilde{p}$) for some $\tilde{p} \in \langle a, p \rangle$.
  \end{enumerate}
\end{corollary}
\begin{proof}
  To see (a), apply Lemma \ref{lem:char_by_quotient_case_pt_pl} in the quotient of $b$
  with $\langle a, b \rangle/b$ for $a$, $\langle a, p, b\rangle/b$ for $\ell$,
  and $\pi/b$ for $\pi$. 
  
  To see (b), apply Lemma \ref{lem:char_by_quotient_case_pt} in the quotient of $b$
  with $\langle a, b \rangle/b$ for $a$ and $\langle a, p, b\rangle/b$ for $\ell$.
  Hence, either $f_b = \tilde{p}^+$ or $f_b = \tilde{p}^+ \lor \tilde{\pi}^+$.
  Suppose that $f_b = \tilde{p}^+ \lor \tilde{\pi}^+$. Then $\pi \cap \tilde{\pi}$
  has dimension at least $n-2 > k$. Hence, we find a $c \in \pi \cap \tilde{\pi}$ such that $c \notin \langle a, b, p \rangle$.
  By (a), $f_c = \tilde{p}^+ \lor \pi^+$ for some $\tilde{p} \in \langle a, p \rangle$ and
  $f_c = \bar{p}^+ \lor \tilde{\pi}^+$ for some $\bar{p} \in \langle b, p \rangle$. This contradicts $\pi \neq \tilde{\pi}$.
  Hence, $f_b = \tilde{p}^+$.
\end{proof}

\begin{lemma}\label{lem:antiflag_compl}
  If $f_a = p^+ \lor \pi^+$ with $a \in \pi$, then $f = \tilde{p}^+ \lor \pi^+$
  for some $\tilde{p} \in \langle a, p \rangle \setminus \{ a \}$.
\end{lemma}
\begin{proof}
  Corollary \ref{cor:antiflag_hull} shows that all $b \notin \langle a, p \rangle$
  satisfy $f_b = \tilde{p}_b^+ \lor \pi^+$ for some $\tilde{p}_b \in \langle a, p \rangle$.
  First we show that we can choose $\tilde{p}_b$ independently of $b$.
  Fix $b$. As $n > 3$, we find a $c \notin \langle a, b, p \rangle$.
  Hence, $f_{c} = p_1^+ \lor \pi^+$ for some $p_1 \in \langle a, p \rangle$ and
  $f_{c} = p_2^+ \lor \pi^+$ for some $p_2 \in \langle b, \tilde{p}_b \rangle$.
  Hence, 
  $\langle p_1 \rangle = \langle p_2 \rangle = \langle a, p \rangle \cap \langle b, \tilde{p}_b \rangle$.
  Here we used that $b \notin \langle a, p \rangle$ implies $\langle a, p \rangle \neq \langle b, \tilde{p}_b \rangle$.
  Hence, we can write $f_a = f_b = f_c = \tilde{p}^+ \lor \pi^+$ if we choose $\tilde{p} \in \langle a, p \rangle \cap \langle b, \tilde{p}_b \rangle$.
  Hence, $\tilde{p}$ is independent of our choice of $b$ and $c$.
  Hence, $f_b = \tilde{p}^+ \lor \pi^+$ for all $b \notin \langle \tilde{p} \rangle$.
  
  It remains to show that $f_{\tilde{p}} = 1$.
  All $b \neq \tilde{p}$ satisfy $f_b = \tilde{p}^+ \lor \pi^+$,
  so $f(K) = 1$ whenever $K$ contains $b$. Since $k > 2$, all $k$-spaces through $\tilde{p}$
  contain such a $b$, and so $f_{\tilde{p}} = 1$ follows.
  
  Hence, $f = \tilde{p}^+ \lor \pi^+$.
\end{proof}

\begin{corollary}\label{cor:point_hull}
  Let $f_a = p^+$ and suppose that $f_b \neq \tilde{p}^+ \lor \tilde{\pi}^+$ 
  for any point $\tilde{p}$ and hyperplane $\tilde{\pi}$ with $b \in \tilde{\pi}$.
  If $b \notin \langle a, p \rangle$, then $f_b = \tilde{p}^+$ for some $\tilde{p} \in \langle a, p \rangle$.
\end{corollary}
\begin{proof}
  Apply Lemma \ref{lem:char_by_quotient_case_pt} in the quotient of $b$
  with $\langle a, b \rangle/b$ for $a$ and $\langle a, p, b \rangle/b$ for $\ell$.
  Hence, $f_b = \tilde{p}^+$, as we are not allowing that $f_b = \tilde{p}^+ \lor \tilde{\pi}^+$ for some hyperplane $\tilde{\pi}$.
\end{proof}

\begin{lemma}\label{lem:point_compl}
  Suppose that there exists no $b$ with $f_b = p^+ \lor \pi^+$ for some point $p$ and hyperplane $\pi$ with $b \in \pi$.
  If $f_a = p^+$, then $f = \tilde{p}^+$
  for some $\tilde{p} \in \langle a, p \rangle \setminus \{ a \}$.
\end{lemma}
\begin{proof}
  Corollary \ref{cor:point_hull} shows that all $b \notin \langle a, p \rangle$
  satisfy $f_b = \tilde{p}_b^+$ for some $\tilde{p}_b \in \langle a, p \rangle$.
  The remaining steps to see $f = \tilde{p}^+$ are identical to the proof of Lemma \ref{lem:antiflag_compl}.
%
%
\end{proof}

\begin{corollary}\label{cor:hyp_hull}
  Suppose that there exists no $b$ with $f_b = p^+ \lor \pi^+$ for some point $p$ and hyperplane $\pi$ with $p \neq b$ and $b \in \pi$.
  If $f_a = \pi^+$ for a hyperplane $\pi$ containing $a$, then $f_b = \pi^+$ for all points $b \in \pi$.
\end{corollary}
\begin{proof}
  Apply Lemma \ref{lem:char_by_quotient_case_pl} in the quotient of $b$
  with $\langle a, b \rangle/b$ for $a$ and $\pi/b$ for $\pi$.
\end{proof}

\begin{lemma}\label{lem:hyp_compl}
  Suppose that there exists no $b$ with $f_b = p^+ \lor \pi^+$ for some point $p$ and hyperplane $\pi$ with $p \neq b$ and $b \in \pi$.
  If $f_a = \pi^+$ for a hyperplane $\pi$ containing $a$, then $f = \pi^+$.
\end{lemma}
\begin{proof}
  By Corollary \ref{cor:hyp_hull}, it is clear that all $b \in \pi$ satisfy $f_b = \pi^+$.
  Let $c$ be a point not in $\pi$. Since $\dim(\pi) = n-1$ and $k \geq 2$, we can find a subspace $K$ through $c$ containing some point $b \in \pi$.
  Then $f_c(K) = 0$, and so $f_c \neq 1$. It follows that $f_c = 0$ (all other cases being ruled out by Lemma~\ref{lem:antiflag_compl} and Lemma~\ref{lem:point_compl}, and their complemented versions).
  Hence, $f = \pi^+$.
\end{proof}

Equivalent statements where all $+$ are replaced by $-$ follow by taking complements.
It remains to show an analog of Lemma \ref{lem:johnson-deg1-inductive}.
\begin{theorem}\label{thm:grassmann_inductive_step}
  Let $n-k, k > 2$.
  If all Boolean degree~$1$ functions on $J_q(n-1,k-1)$ are trivial, then all Boolean degree~$1$ functions on $J_q(n,k)$ are also trivial.
\end{theorem}
\begin{proof}
  This immediately follows from Lemma \ref{lem:antiflag_compl}, Lemma \ref{lem:point_compl}, Lemma \ref{lem:hyp_compl}, and their complemented versions.
\end{proof}


\begin{proof}[Proof of Theorem \ref{thm:grassmann_class}]
  Due to Theorem \ref{thm:k_eq_2_results}, all Boolean degree~$1$ functions on $J_q(n, 2)$ (and so, by duality, on $J_q(n,n-2)$) are trivial for $n \geq 5$ and $q \in \{ 2, 3, 4, 5 \}$.
  Applying Theorem~\ref{thm:grassmann_inductive_step} inductively, we deduce that all Boolean degree~$1$ functions on $J_q(n,k)$ are trivial whenever $k,n-k > 2$ and $q \in \{2,3,4,5\}$.
\end{proof}

\section{Dual polar graphs and related graphs}\label{ch:polar}

Let $q$ be a prime power.
A \textit{polarity} $\perp$ of $\FF_q^m$ is a bijective map on all subspaces of $\FF_q^m$ 
such that a $k$-space is mapped to an $(m-k)$-space which is incidence preserving, that is $A \subseteq B \rightarrow B^\perp \subseteq A^\perp$ 
and an involution, that is $\perp^2 = 1$.
We say that a subspace with $L \subseteq L^\perp$ is \textit{isotropic}.
A polar space consists of all isotropic subspaces except for some cases with $q$ even \footnote{See for example \cite{Brouwer1989} or \cite{Terwilliger} for a formal definition in these cases.}
The finite classical polar spaces are $O^+(2n, q)$, 
$O(2n+1, q)$, $O^-(2n+2, q)$, $Sp(2n, q)$, $U(2n, q)$, and $U(2n+1, q)$. 
We associate (in the same order) the parameter $e = 0, 1, 2, 1^*, 1/2, 3/2$
with them. 
Here $n$ is the \textit{rank} of the polar space, that is the 
maximal dimension of a isotropic subspace.
Note that $q$ is a square of a prime power for $U(2n, q)$ and $U(2n+1, q)$.
To simplify notation, we denote a polar space of rank $n$ with parameter $e$
over $\FF_q$ by $\PP_q(n, e)$.
The natural embedding in $\FF_q^m$ is as follows: $O^+(2n, q)$, $Sp(2n, q)$ and $U(2n, q)$ in $\FF_q^{2n}$,
$O(2n+1, q)$ and $U(2n+1, q)$ in $\FF_q^{2n+1}$, and $O^-(2n+2, q)$ in $\FF_q^{2n+2}$.
We call the corresponding vector space the \textit{ambient space}.
We call a subspace of maximal dimension $n$ a \textit{maximal}.
The graphs $C_q(n, k, e)$ are the graphs having as vertices the
isotropic $k$-spaces of $\PP_q(n, e)$, two vertices being adjacent when their meet is a subspace of dimension $k-1$.
The graph $C_q(n, n, e)$ is known as the \textit{dual polar graph} of $\PP_q(n, e)$,
while the graph $C_q(n, 1, e)$ is sometimes called the \textit{collinearity graph}
of a $\PP_q(n, e)$.
We identify $\PP_q(n, e)$ with its natural embedding in $\FF_q^m$ for some $m \in \{ 2n, 2n+1, 2n+2 \}$,
so when $S$ and $T$ are subspaces of $\PP_q(n, e)$, then $\langle S, T \rangle$ is a subspace
of $\FF_q^m$ and $\PP_q(n, e) \cap \langle S, T \rangle$ is a subgeometry of $\PP_q(n, e)$.

The eigenspaces of $C_q(n, k, e)$ were described by Eisfeld \cite{Eisfeld1999} 
(see also \cite{Stanton1980a} and \cite[\S4.2]{Vanhove2011}). From this it is clear that 
for $C_q(n, n, e)$, degree $1$ functions correspond to completely regular strength $0$
codes of covering radius 1. For some combinations of $e$ and $n$, these are called 
Cameron--Liebler sets in polar spaces of type I in \cite{DeBoeck2017}.
A similar observation can be made for all $C_q(n, k, e)$ with $k>1$. There, the 
span of the indicator functions $p_i^+$ for all points $p_i$ span three eigenspaces, for example in 
Vanhove's notation for $n\geq k+1$, these are $V_{0,0}^k + V_{1,0}^k + V_{1,1}^k$.

The graphs $C_q(n, 1, e)$ are strongly regular and it seems to be most reasonable to consider maximals
(rather than points) as our coordinates as (in the usual ordering)
the subspace $V_0 + V_1$ corresponds to the span of the characteristic vector of the maximals.
We want to point out that Boolean degree $1$ functions on $C_q(n, 1, e)$
are commonly known as tight sets \cite{Bamberg2007}, and have been intensively
investigated. The rest of this section is concerned only with $k>1$.
Tight sets have an important role for the $J_q(4, 2)$ case of the Grassmann
graph, as $J_q(4, 2)$ and $C_q(3, 1, 0)$ are isomorphic by the Klein correspondence
between $\FF^3$ and $O^+(6, \FF)$.

\subsection{Some properties}

In this short section, we list a few useful properties of Boolean degree $1$ functions
on $C_q(n, k, e)$ for $k>1$.
We already saw that a polar space $\PP_q(n, e)$ has a natural embedding 
in a vector space $\FF_q^m$. Let $\pi$ be a hyperplane of $\FF_q^m$.
Then $\PP_q(n, e) \cap \pi$ is either \textit{degenerate}, that is isomorphic 
to $p\PP_q(n-1, e)$ for some isotropic point $p$, or \textit{non-degenerate}.
The possibilities for the type of a non-degenerate hyperplane intersection
are as follows:
\begin{enumerate}[(a)]
 \item A non-degenerate hyperplane of $O^+(2n, q)$ is isomorphic to $O(2n-1, q)$,
 \item a non-degenerate hyperplane of $O(2n+1, q)$ is isomorphic to $O^+(2n, q)$ or $O^-(2(n-1)+2, q)$,
 \item a non-degenerate hyperplane of $O^-(2n+2, q)$ is isomorphic to $O(2n+1, q)$,
 \item a non-degenerate hyperplane of $U(2n+1, q)$ is isomorphic to $U(2n, q)$,
 \item a non-degenerate hyperplane of $U(2n, q)$ is isomorphic to $U(2n-1, q)$.
\end{enumerate}
The polar space $Sp(2n, q)$ has no non-degenerate hyperplane sections, but for 
$q$ even it is isomorphic to $O(2n+1, q)$.

We observe the following:
\begin{enumerate}[(a)]
 \item $C_q(n, k, 0)$, $C_q(n, k, 1^*)$, and $C_q(n, k, 1/2)$ are coordinate-induced subgraphs of $J_q(2n, q)$,
 \item $C_q(n, k, 1)$ and $C_q(n, k, 3/2)$ are coordinate induced subgraphs of $J_q(2n+1, k)$,
 \item $C_q(n, k, 2)$ is a coordinate-induced subgraph of $J_q(2n+2, k)$,
 \item $C_q(n, k, 0)$ is a coordinate-induced subgraph of $C_q(n, k, 1)$, $C_q(n, k, 1)$ 
 of $C_q(n, k, 2)$, and $C_q(n, k, 1/2)$ of $C_q(n, k, 3/2)$.
\end{enumerate}

\subsection{Some results}

Some Boolean degree $1$ functions of $C_q(n, k, e)$ are induced 
from the trivial functions on the ambient graph $J_q(m, k)$.
Here is a complete list:
\begin{enumerate}[(a)]
 \item $f = 1^\pm$.
 \item $f = p^\pm$ for a point $p$.
 \item $f = \pi^\pm$ for a hyperplane $\pi$ of the ambient space. 
      Notice that $\pi$ can intersect the polar space either in a proper polar space
      or in $p^\perp$ for some point $p$ of $\PP_q(n, e)$.
 \item $f = (p \lor \pi)^\pm$, where $p$ and $\pi$ are as before.
\end{enumerate}
Notice that Boolean degree $1$ polynomials expressed with these functions
can be very complicated. For example, if $\pi = p^\perp$,
then $f = \pi^+ \land p^-$ is a Boolean degree $1$ function.

We conjecture the following:
\begin{conjecture}\label{conj:polar_main}
  Let $k \geq 2$. Then there exists $n_{q,k}$ such that
  every Boolean degree $1$ function $f$ on $C_q(n, k, e)$ with $n \geq n_{q,k}$
  is trivial, that is $f$ can be written as
  \begin{align*}
    f^\pm = \bigvee_i p_i^+ \vee \bigvee_i \pi_i^+ \vee \bigvee_i (\tilde{\pi}_i^+ \wedge \tilde{p}_i^-).
  \end{align*}
  Here $p_i$ are the points of the polar space, $\pi_i$ are non-degenerate hyperplanes
  of the ambient space, and $\tilde{\pi}_i$ are degenerate hyperplanes of the ambient space
  with $\tilde{\pi}_i = \tilde{p}_i^\perp$.
\end{conjecture}
The conjecture simplifies in some cases:
\begin{enumerate}
 \item For $k < n$ every pair of hyperplanes contains a common $k$-space,
and so the condition simplifies in this case, and we have to use at most one $\pi_i$ or $\tilde{\pi}_i$.
 \item For $k=n$ no non-degenerate hyperplane of $O^+(2n, q)$ or $U(2n, q)$ contains a $k$-space,
 so no $\pi_i$ occur.
 \item For $k=n$ we have $\tilde{p}_i^\pm = \tilde{\pi}_i^\pm$, so the last clause in the conjecture is superfluous. 
 \item $Sp(2n, q)$ has only degenerate hyperplane sections. If $q$ is even, then we identify $Sp(2n, q)$ with the isomorphic
    $O(2n+1, q)$ which has non-degenerate hyperplane sections. If $q$ is odd, then again no $\pi_i$ occur.
\end{enumerate}
We prove Conjecture \ref{conj:polar_main} for one particular case:

\begin{theorem}\label{thm:polar_small_dim_main}
  Let $k, n-k \geq 2$ or $(n,k) \in \{ (2,2), (3,2) \}$. Then
  every Boolean degree $1$ function $f$ on $C_2(n, k, 0)$
  is trivial, that is $f$ or $1-f$ is one of the following:
  \begin{enumerate}[(a)]
   \item $f = 1^\pm$,
   \item $f = \pi^\pm$ for a hyperplane $\pi$,
   \item $f = (\bigvee_i p_i^+)^\pm$ for pairwise non-collinear points $p_i$,
   \item $f = (\pi^+ \vee \bigvee_i p_i^+)^\pm$ for pairwise non-collinear points $p_i$
   and a hyperplane $\pi$, where $p_i \notin \pi$,
   \item $f = ((\pi^+ \land p_1^-) \vee \bigvee_{i>1} p_i^+)^\pm$ 
   for pairwise non-collinear points $p_i$, $\pi$ a degenerate hyperplane,
   and $p_1 = \pi^\perp$.
  \end{enumerate}
\end{theorem}

Using the mixed integer program solver Gurobi, we obtain the following.
\begin{lemma}\label{lem:polar_small_dim_base_case}
  Theorem \ref{thm:polar_small_dim_main} holds for $(n,k)=(2,2)$, $(n,k)=(3,2)$ and $(n,k) = (4, 2)$.
\end{lemma}
Notice that for the case $(n,k)=(4,2)$ it is helpful to use Lemma \ref{lem:reduced_and_trivial}
for additional constraints.

\begin{lemma}\label{lem:polar_pt_trans}
  Let (a) $n \geq 5$ or (b) $n=4$ and $q=2$. Let $k, n-k \geq 2$.
  Let $f$ be a Boolean degree $1$ function of $C_q(n, k, e)$.
  Let $S$ be a maximal of $\PP_q(n, e)$ and $p$ a point of $S$.
  Then the following holds:
  \begin{enumerate}[(a)]
   \item $f_S = p^+$ implies that $f_T = p^+$ for all maximals $T$ with $p \in T$,
   \item $f_S = p^+ \lor \pi_S^+$ implies 
	that $f_T = p^+ \lor \tilde{\pi}_T^+$ for all maximals $T$ with $p \in T$ ($\pi_X$ a hyperplane of $X \in \{ S, T \}$).
  \end{enumerate}
\end{lemma}
\begin{proof}
  As the maximals in quotient of $p$ are connected with respect to intersecting in a hyperplane, it is sufficient to show
  that $p^+ \rightarrow f_T$ for $\dim(S \cap T) = n-1$.
  As $S \cap T \subseteq S$, $f_{S \cap T} = p^+$ or $f_{S \cap T} = p^+ \lor \tilde{\pi}^+$ 
  for some hyperplane $\tilde{\pi}$ of $S \cap T$.
  By Theorem \ref{thm:grassmann_class}, $f_T$ is trivial.
  The dual versions of Lemma \ref{lem:antiflag_compl} and Lemma \ref{lem:hyp_compl} imply that
  the only trivial Boolean degree $1$ function $f'$ in $J_q(n, k)$ which contains 
  an $(n-1)$-space $R$ with $f_R' = p^+$, respectively, $f_R' = p^+ \lor \tilde{\pi}^+$,
  is $p^+$, respectively, $p^+ \lor \pi^+$, where $\tilde{\pi} \subseteq \pi$.
\end{proof}

We can do the following two step process:
\begin{enumerate}[(1)]
\item By Lemma \ref{lem:polar_pt_trans}, $p^+ \rightarrow f_S$, $f_S \neq 1$, implies that $p^+ \rightarrow f_{p^\perp}$ and $f_{p^\perp} \neq 1$.
Hence, we can replace $f$ by $f \land p^-$ and still have a Boolean degree $1$ function.
As this reduces $|f|$ in every step, we can do this till no such $S$ occurs anymore.
\item By Lemma \ref{lem:polar_pt_trans} (and taking complements), $p^- \rightarrow f_S$, $f_S \neq 0$, 
imply that $p^- \rightarrow f_{p^\perp}$ and $f_S \neq 0$.
Again, we can replace $f$ by $f \lor p^+$ till no such $S$ occurs anymore.
\end{enumerate}
If there is no maximal $S$ left such that $f_S$ is not constant and $p^\pm \rightarrow f_S$,
then we call $f$ \textit{reduced}.
On the restriction to $J_q(n, k)$ it is easily seen that all $p$'s used in the process are disjoint:
suppose that $f_S = p^+$ before Step (1) and $f_S = \tilde{p}^-$ after Step (1).
But first we had $f_S = p^+$, so afterwards we have $f_S = 0 \neq \tilde{p}^-$.
Similarly, $f_S = p^+ \lor \tilde{\pi}^+$ before Step (1) implies $f_S = \tilde{\pi}^+$
after Step (1). Again, this is different from $f_S = \tilde{p}^-$.
Hence, a non-constant trivial reduced Boolean degree $1$ function is $\pi^\pm$ 
for some hyperplane $\pi$.

We repeatetly use the following property:
\begin{lemma}\label{lem:reduced_and_trivial}
  Let $k, n-k \geq 2$. Let $f$ be a reduced Boolean degree $1$ function on $C_q(n, k, e)$
  and suppose that $f_S$ is trivial for all maximals $S$. 
  \begin{enumerate}[(a)]
   \item If all Boolean degee $1$ functions $g$ induced by $f$ on $C_q(n-1, k, e)$ are trivial,
   then $g$ is reduced.
   \item If all Boolean degee $1$ functions $g$ induced by $f$ on $C_q(n-1, k-1, e)$ are trivial,
   then $g$ is reduced.
  \end{enumerate}
\end{lemma}
\begin{proof}
  We do know that $f_S$ is trivial and reduced, so $f_S \in \{ 0, 1, \pi^+, \pi^-\}$
  for some hyperplane of $\pi$ of $S$. For $k, n-k \geq 2$, this implies $f_T \in \{ 0, 1, \pi^+, \pi^- \}$
  for every hyperplane $T$ of $S$. Hence, $g_T$ is reduced for all maximals of $C_q(n-1, k, e)$.
  Hence, $g$ is reduced by Lemma \ref{lem:polar_pt_trans}. This shows (a). Part (b) follows similarly.
\end{proof}

\begin{lemma}\label{lem:polar_hyperplane_small_dim}
  Let $n-1 > k > 1$.
  Let $f$ be a reduced trivial Boolean degree $1$ function on $C_q(n, k, e)$.
  Let $\pi$ be a hyperplane of the ambient space of $\PP_q(n, e)$, and fix an isotropic point $a \in \pi$.
  Suppose that all $k$-spaces $K$ through $a$ satisfy
  \begin{enumerate}[(a)]
   \item $f(K) = 1$ if $K \subseteq \pi$,
   \item $f(K) = 0$ otherwise.
  \end{enumerate}
  Then $f = \pi^+$.
\end{lemma}
\begin{proof}
  As not all $k$-spaces $K$ through $a$ satisfy $f(K) = 1$, we have $f \notin \{ 0, 1 \}$.
  As $f$ is reduced, trivial, and $k < n$, the only options left are $f = \tilde{\pi}^\pm$ 
  for some hyperplane $\tilde{\pi}$ of the ambient space.
  
  Suppose that $a \notin \tilde{\pi}$. Then no isotropic $k$-space through $a$ lies in $\tilde{\pi}$.
  Hence, if $f = \tilde{\pi}^+$, $f(K) = 0$ for isotropic $k$-spaces $K$ through $a$,
  and if $f = \tilde{\pi}^-$, $f(K) = 1$ for isotropic $k$-spaces $K$ through $a$.
  As $f_a$ is not constant, this is a contradiction. Hence, $a \in \tilde{\pi}$.
  
  Suppose that $f = \tilde{\pi}^-$. Let $S$ be a maximal through $a$.
  As $\dim(\tilde{\pi} \cap \pi \cap S) \geq \dim(S)-2 = n-2 \geq k$,
  there exists an isotropic $k$-space $K$ through $a$ which lies in $\tilde{\pi}$ and $\pi$.
  As $a \in K \subseteq \pi$, $f(K) = 1$. As $f = \tilde{\pi}^-$, $f(K) = 0$ which is a contradiction.
  
  Suppose that $f = \tilde{\pi}^+$ with $\pi \neq \tilde{\pi}$.
  Then there exists a maximal isotropic subspace $S$ of $\pi \cap \PP_q(n, e)$ (which can have dimension $n$ or $n-1$)
  through $a$ in $\pi$ which is not contained in $\tilde{\pi}$ (as $\tilde{\pi}$ is a proper subspace pf $a^\perp$).
  As $\dim(S) \geq k-1$, there exists an isotropic $k$-space $K$ in $S$ through $a$ with $K \nsubseteq \tilde{\pi}$.
  As $a \in K \subseteq S \subseteq \pi$, $f(K) = 1$. As $f = \tilde{\pi}^+$, $f(K) = 0$.
  This is a contradiction. Hence, $\pi = \tilde{\pi}$.
\end{proof}

\begin{theorem}\label{thm:polar_inductive_step_inc_k}
  Let $k, n-k \geq 2$.
  If all Boolean degree $1$ functions on $C_q(n, k, e)$ and $J_q(n+1, k+1)$ are trivial,
  then all Boolean degree $1$ functions on $C_q(n+1, k+1, e)$ are trivial.
\end{theorem}
\begin{proof}
  If $f_a \in \{ 0, 1 \}$ for all points $a$ of $\PP_q(n+1, e)$, 
  then $f$ is constant and we are done.
  Hence, there exists a point $a$ such that $f_a$ is not constant.
  Furthermore, we can assume that $f$ is reduced, so $f_a = \pi^\pm$
  for some hyperplane through $a$.
  
  By considering $1-f$ instead of $f$, we can without loss of generality assume that $f_a = \pi^+$.
  First let $b$ be an isotropic point in $a^\perp \cap \pi$.
  Set $\ell = \langle a, b \rangle$. 
  As all Boolean degree $1$ functions on $C_q(n, k, e)$ are trivial, $f_b$ is trivial,
  and by Lemma \ref{lem:reduced_and_trivial}, reduced.
  If $b \in \pi$, then, by Lemma \ref{lem:polar_hyperplane_small_dim} and the assumption that
  all Boolean degree $1$ functions on $C_q(n, k, e)$ are trivial, $f_b = \pi^+$.
  As for $n \geq 2$ all pairs of points $(p_1, p_2)$ in $\pi$ have a common neighbour in $\pi$
  (that is $p_1^\perp \cap p_2^\perp \cap \pi$ contains an isotropic point), this implies that
  $f_c = \pi^+$ for all isotropic $c \in \pi$.
  
  Now let $b$ be an isotropic point not in $\pi$. 
  By Lemma \ref{lem:reduced_and_trivial}, $f_b \in \{ 0, 1, \tilde{\pi}^+, \tilde{\pi}^- \}$.
  We want to show that $f_b = 0$ by ruling out the three other cases.
  As $n \geq 4$, $b^\perp \cap \pi \cap \tilde{\pi}$ contains an isotropic point $c$.
  Set $\ell = \langle b, c \rangle$.
  As $f_c = \pi^+$ and $b \notin \pi$, we have $f_{\ell} = 0$,
  so we can rule out that $f_b = 1$. As $c^\perp \cap \tilde{\pi}$ is a hyperplane of $\tilde{\pi}$
  and $n - 2 \geq k$,
  we find an isotropic $k$-space $L$ in  $\tilde{\pi}$ which contains $b$ and $c$.
  As $f_{\ell}(L) = 0$, we rule out $f_{b} = \tilde{\pi}^+$.
  Choosing $c \in \pi \setminus \tilde{\pi}$ instead of $\pi \cap \tilde{\pi}$
  rules out $f_b = \tilde{\pi}^-$ with a similar argument.
  Hence, $f_b = 0$.
  
  This concludes that $f = \pi^+$.
\end{proof}

\begin{lemma}\label{lem:polar_from_a_to_aperp}
  Let $k, n-k-1 \geq 2$.
  Let $f$ be a reduced Boolean degree $1$ function on $C_q(n+1, k, e)$
  such that for all non-degenerate hyperplanes in $\tau \subseteq a^\perp$ the function $f_\tau$ is trivial.
  Then $f_{a^\perp} \in \{ 0, 1, \pi^+, \pi^- \}$, where $\pi^+$ is some hyperplane of $a^\perp$.
\end{lemma}
\begin{proof}
  We first show the claim for $f_{a^\perp \setminus a}$ instead of $f_{a^\perp}$.
  Here we assume that $f_{a^\perp}$ is not constant and distinguish two cases.
  
  Case 1: Suppose that $f_\tau = 1$ for some non-degenerate hyperplane $\tau \subseteq a^\perp$.
  Let $\tilde{\tau}$ be another non-degenerate hyperplane of $a^\perp$.
  Clearly, $(\tau \cap \tilde{\tau})^+ \rightarrow f_{\tilde{\tau}}$
  and, by Lemma \ref{lem:reduced_and_trivial}, $f_{\tilde{\tau}} \in \{ 0, 1, \tilde{\pi}^+, \tilde{\pi}^- \}$ for some hyperplane $\tilde{\pi}$
  of $\tilde{\tau}$.
  We claim that $f_{\tilde{\tau}} = (\tau \cap \tilde{\tau})^+$. 
  As $n - 3 \geq k$, $\pi \cap \tilde{\pi}$ contains an
  isotropic $k$-space, so $f_{\tilde{\tau}} \neq 0, \tilde{\pi}^-$.
  It remains to rule out $f_{\tilde{\tau}} = 1$.
  If $f_{\tilde{\tau}} = 1$, then let $\overline{\tau}$ be a 
  non-degenerate hyperplane of $a^\perp$ which does not contain $\tau \cap \tilde{\tau}$.
  As $\tau \cap \overline{\tau} \neq \tilde{\tau} \cap \overline{\tau}$,
  following the same arguments as for $f_{\tilde{\tau}}$, we obtain that $f_{\overline{\tau}} = 1$.
  As we can permute the roles of $\tau$, $\tilde{\tau}$ and $\overline{\tau}$,
  it follows that $f_{\overline{\tau}} = 1$ for all non-degenerate hyperplanes of $a^\perp$.
  Hence, $f_{a^\perp \setminus a} = 1$ which contradicts our assumption that $f_{a^\perp \setminus a} \neq 1$.
  Hence, $f_{\overline{\tau}} = \tau^+$ for all $\tau \neq \overline{\tau}$.
  
  Case 2: No non-degenerate hyperplane $\tau$ of $a^\perp$ satisfies $f_{\tau} = 1$.
  Then we can assume that for a non-degenerate hyperplane $\tau$ of $a^\perp$ we find a hyperplane $\pi$ of $\tau$
  with $f_{\tau} = \pi^+$. 
  Set $\pi' = \langle a, \pi \rangle^+$.
  We claim that $f_{a^\perp \setminus a} = \pi'^+$ follows.
  Clearly, $f_{\tilde{\tau}} = \pi^+$ for all non-degenerate hyperplanes $\tilde{\tau}$ through $\pi$
  (here we use $n-3 \geq k$ to rule out that $f_{\tilde{\tau}} = \tilde{\pi}^-$ for some hyperplane $\tilde{\pi}$ of $\tilde{\tau}$).
  As all isotropic $k$-spaces $K \subseteq a^\perp \setminus a$, which are not in $\pi'$, lie in one such $\tilde{\tau}$, it follows
  that $f_{a^\perp \setminus a}(K) = 0$. This rules out $f_{\overline{\tau}} = \overline{\pi}^-$
  for all non-degenerate hyperplanes $\overline{\tau}$ (and some hyperplane $\overline{\pi}$ 
  of $\overline{\tau}$). As $n - 2 \geq k$, $\overline{\tau} \cap \pi$ contains at least one isotropic $k$-space,
  so $f_{\overline{\tau}} \neq 0$. Hence, $f_{\overline{\tau}} = \pi'^+$.
  
  We have seen that $f_{a^\perp \setminus a}$ is trivial. By Corollary \ref{cor:hyp_hull},
  $f$ reduced, $k, n-k \geq 2$ and looking at all maximals through $a$, we 
  conclude that $f_{a^\perp}$ is trivial.
\end{proof}

\begin{theorem}\label{thm:polar_inductive_step}
  Let $k, n-k \geq 2$.
  If all Boolean degree $1$ functions on $C_q(n, k, e)$ and $J_q(n+1, k)$ are trivial,
  then all Boolean degree $1$ functions on $C_q(n+1, k, e)$ are trivial.
\end{theorem}
\begin{proof}
  We can assume that $f$ is reduced and not constant.
  
  First suppose that $f_{a^\perp} = 1$ for some isotropic point $a$.
  Our claim is that this implies $f = (a^\perp)^+$.
  Clearly, $(a^\perp)^+ \rightarrow f_{b^\perp}$ for all isotropic points $b$.
  By Lemma \ref{lem:polar_from_a_to_aperp}, $f_{b^\perp}$ is trivial, 
  so $f_{b^\perp} \in \{ (a^\perp)^+, \pi^-, 1 \}$ for some hyperplane $\pi$ of $b^\perp$.
  As $n-2 \geq k$, $\pi \cap a^\perp$ contains a $k$-space for all hyperplanes $\pi$,
  so $f_{b^\perp} = \pi^-$ does not occur. Now suppose that $f_{b^\perp} = 1$.
  We will show that this implies $f = 1$ which contradicts the assumption that $f$
  is not constant. To see this, consider an isotropic point $c$ such that $c^\perp$ does not contain
  $a^\perp \cap b^\perp$. Then $(a^\perp)^+ \rightarrow f_{c^\perp}$ and $(b^\perp)^+ \rightarrow f_{c^\perp}$.
  We chose $c$ such that $(a^\perp \cap c^\perp)^+ \neq (b^\perp \cap c^\perp)^+$, so, as $f_{c^\perp}$ is trivial, this implies
  $f_{c^\perp} = 1$. By permuting the roles of $a$, $b$, and $c$, we obtain that $f_{c^\perp} = 1$ for all isotropic points
  $c$. Hence, $f=1$ which contradicts $f$ non-constant. Therefore $f = (a^\perp)^+$.
  
  Now suppose that $f_{a^\perp} \neq 1$ for all isotropic points $a$.
  As $f$ is non-constant and reduced, we can assume without loss of generality that $f_{a^\perp} = \pi^+$
  for a hyperplane $\pi$ of $a^\perp$. 
  Let $b$ be an isotropic not in $a^\perp$ with $\pi \nsubseteq b^\perp$.
  As before, $f_{b^\perp} \in \{ 0, 1, \tilde{\pi}^+, \tilde{\pi}^- \}$
  for some hyperplane $\tilde{\pi}$ of $b^\perp$. 
  Our goal is to show that $f_{b^\perp} = \tilde{\pi}^+$ and then $f = \langle \pi, \tilde{\pi}\rangle^+$.
  By Lemma \ref{lem:polar_from_a_to_aperp}, we have $f_{b^\perp} \in \{ 0, \tilde{\pi}^+, \tilde{\pi}^- \}$.
  As $n - k \geq 2$, $b^\perp \cap \pi$ (respectively, $\tilde{\pi} \cap \pi$) contains at least one 
  isotropic $k$-space, so $f_{b^\perp} = \tilde{\pi}^+$ (with $\dim(\pi \cap \tilde{\pi}) = n-1$).
  Now let $c$ be an isotropic point with $\pi \cap \tilde{\pi} \nsubseteq c^\perp$.
  By the same reason as for $b$, $f_{c^\perp} = \overline{\pi}^+$.
  Let $m+1$ be the dimension of the ambient vector space.
  We claim that $\overline{\pi} = \langle \pi \cap \overline{\pi}, \tilde{\pi} \cap \overline{\pi} \rangle$.
  As $\overline{\pi} \nsubseteq \pi \cap \tilde{\pi}$ and $\dim(\pi \cap \overline{\pi}) = \dim(\tilde{\pi} \cap \overline{\pi}) = m-1$ and
  \begin{align*}
   \dim(\pi \cap \tilde{\pi} \cap \overline{\pi}) \leq m-2,
  \end{align*}
  we indeed conclude $\dim(\langle \pi \cap \overline{\pi}, \tilde{\pi} \cap \overline{\pi} \rangle) \geq 2(m-1) - (m-2) = m = \dim(\overline{\pi})$.
  Again, by permuting the roles of $a$, $b$, and $c$ we obtain that $f = \langle \pi, \tilde{\pi} \rangle^+$ which shows that $f$ is trivial.
\end{proof}

\begin{proof}[Proof of Theorem \ref{thm:polar_small_dim_main}]
  We prove the assertion by induction on $n$.
  By Lemma \ref{lem:polar_small_dim_base_case}, we can assume that $n \geq 5$.
  Theorem \ref{thm:polar_inductive_step_inc_k} and Theorem \ref{thm:polar_inductive_step} then complete the proof.
\end{proof}

\section{Sesquilinear forms graphs}

We denote the bilinear forms graphs of $\ell \times k$ bilinear forms over $\FF_q$ by $H_q(\ell, k)$ (we assume $\ell \leq k$),
the alternating forms graphs of $n\times n$ alternating forms over $\FF_q$ by $A_q(n)$, 
the Hermitian forms graphs of $n\times n$ Hermitian forms over $\FF_q$, $q$ a square of a prime power, by $Q_q(n)$,
and the symmetric bilinear forms graphs of $n \times n$ symmetric bilinear forms over $\FF_q$ by $S_q(n)$.
Hua observed \cite{Hua1949} (see also \cite[\S9.5E]{Brouwer1989})
that bilinear forms graphs, alternating forms graphs, Hermitian forms graphs, and symmetric bilinear forms graphs are induced subgraphs of either Grassmann graphs or certain dual polar
graphs. In particular:
\begin{enumerate}[(a)]
 \item Let $L$ be an $\ell$-space of $\FF_q^{k+\ell}$. Let $Y$ denote all subspaces
      of $J_q(n, k)$ disjoint from $L$. Then $H_q(\ell, k)$ is the induced subgraph of
      $J_q(\ell+k, k)$ on $Y$.
 \item Let $x$ be a maximal of $\PP_q(n, 1/2)$. Let $Y$ denote all maximals
      of $\PP_q(n, 1/2)$ disjoint from $x$. Then $Q_q(n)$ is the induced subgraph of
      $C_q(n, n, 1/2)$ on $Y$.
 \item Let $x$ be a maximal of $\PP_q(n, 1^*)$. Let $Y$ denote all maximals
      of $\PP_q(n, 1^*)$ disjoint from $x$. Then $S_q(n)$ is the induced subgraph of
      $C_q(n, n, 1^*)$ on $Y$.
 \item Let $x$ be a maximal of $\PP_q(n, 0)$. Let $Y$ denote all maximals
      of $\PP_q(n, 0)$ disjoint from $x$. Then $A_q(n)$ is the induced subgraph of
      the distance-$2$-graph of $C_q(n, n, 0)$ on $Y$.
\end{enumerate}
All these graphs are coordinate-induced subgraphs.

The eigenspaces of the bilinear forms graphs were described by Delsarte \cite{Delsarte1978},
and, using the identification of $H_q(\ell, k)$ with a subgraph of $J_q(\ell+k, k)$
(see \cite[\S9.5A]{Brouwer1989}), it is easy to see that Boolean degree $1$ functions of
$H_q(\ell, k)$ correspond to completely regular strength $0$ codes of covering radius $1$ of $H_q(\ell, k)$.
See the appendix of \cite{Schmidt2018} for an explicit discussion of the eigenspaces of $Q_q(n)$.
For $S_q(n)$ and $A_q(n)$ the same holds, but we are not aware of any reference 
that is more explicit than \cite[\S9.5]{Brouwer1989}.

As already the Boolean degree $1$ functions on $C_q(n, k, e)$ seems to have a very complicated
description (see Conjecture \ref{conj:polar_main}), and we were only able to solve one particular case (see Theorem \ref{thm:polar_small_dim_main}),
we leave the investigation of Boolean degree $1$ functions on the alternating, Hermitian, and symmetric bilinear forms graphs
for future work. In the rest of the section we investigate $H_q(\ell, k)$.

Let us start with a list of examples that are induced from $p^\pm$ and $\pi^\pm$ on $J_q(k+\ell, k)$.
Recall that we identify $H_q(\ell, k)$ with all $k$-spaces of $J_q(k+\ell, k)$ disjoint from a fixed 
$\ell$-space $L$.
\begin{enumerate}[(a)]
 \item $f = p^\pm$ for a point $p \notin L$,
 \item $f = \pi^\pm$ for a hyperplane $\pi$ with $L \nsubseteq \pi$.
\end{enumerate}

Our conjecture is the following:
\begin{conjecture}\label{conj:bilinear_main}
  Let $q$ be a prime power.
  Let $f$ be a Boolean degree $1$ function on $H_q(\ell, k)$, where $k+\ell$ is sufficiently large (depending on $q$).
  Then there exists an $(\ell+1)$-space $g \supseteq L$,
  an $(\ell-1)$-space $G \subseteq L$, points $p_i \in g$,
  and hyperplanes $\pi_i$ with $\pi_i \cap L = G$, and $p_j \notin \pi_i$
  for all $p_i$ and $\pi_j$
  such that
  \begin{align*}
    f^\pm = \bigvee p_i^+ \vee \bigvee \pi_i^+.
  \end{align*}
\end{conjecture}
It is clear that $p_i^+ + p_j^+$ is a Boolean function if and only if 
$\langle p_i, p_j \rangle \cap L$ is a point. Similarly, it is clear that 
$\pi_i^+ + \pi_j^+$ is a Boolean function if and only if $\pi_i \cap L = \pi_j \cap L$.
Hence, Conjecture \ref{conj:bilinear_main} covers exactly the examples
which are induced by examples from $J_q(n, k)$.
For $J_q(4, 2)$ we do know many non-trivial examples for Boolean degree $1$ functions,
so the condition $k+\ell > 4$ is necessary.
By computer we verified the conjecture for $H_2(2, 2)$ and $H_2(2, 3)$.

We start by describing the example by Bruen and Drudge \cite{Bruen1999} for $J_q(4, 2)$, $q$ odd.
Consider the polar space $O^-(4, q)$ in its natural embedding in $\FF_q^4$.
Let $\cP$ denote the isotropic points of $O^-(4, q)$.
A \textit{tangent} is a line of $\FF_q^4$ which contains exactly one point of $\cP$,
a \textit{secant} a line which contains exactly two points of $\cP$,
a \textit{passant} a line which contain no point of $\cP$. It is well-known that 
every line in $\FF_q^4$ is a passant, a tangent, or a secant.
For each point $p \in \cP$, let $\cL_p$ denote a special set of 
$(q+1)/2$ of all $q+1$ tangents of through $p$. We refer the reader to \cite[\S3]{Bruen1999}
for details on how $\cL_p$ is chosen.
Let $\cS$ denote all secants of $\cP$.
\begin{example}[Bruen and Drudge]\label{ex:bruendrudge}
  Let $f_{BD}$ be a Boolean function on $J_q(4, 2)$, $q$ odd, 
  defined by $f_{BD}(K) = 1$ if and only if $K \in \cS \cup \bigcup \cL_p$.
  Then $f_{BD}$ is a Boolean degree $1$ function with $|f_{BD}| = (q^2+1)(q^2+q+1)/2$.
\end{example}
Choose our embedding of $H_q(2, 2)$ in $J_q(4, 2)$ such that 
$L$ is an external line of $O^-(4, q)$.
Let $f_{BD'}$ denote the restriction of $f_{BD}$ to $H(\ell, k)$.
\begin{lemma}\label{lem:BD_on_bilinear}
  We have $|f_{BD'}| = (q^2+1)q^2/2$.
\end{lemma}
\begin{proof}
  An external line $L$ meets exactly $(q^2+1)(q+1)/2$ lines $K$
  with $f(K) = 1$ (e.g. see \cite[Theorem 1.1 (vii) and Theorem 3.1]{Bruen1999}).
  Hence, the claim follows from $|f_{BD}| = (q^2+1)(q^2+q+1)/2$.
\end{proof}

\begin{theorem}\label{lem:bilinear_sizes_canon}
  There exists no Boolean degree $1$ function $f$ on $H_q(2, 2)$ as in Conjecture \ref{conj:bilinear_main}
  such that $f = f_{BD'}$.
\end{theorem}
\begin{proof}
  Suppose that $\pi^\pm \rightarrow f_{BD'}$ for some hyperplane $\pi$ with $L \nsubseteq \pi$.
  
  If $\pi \cap O^-(4, q)$ is isomorphic to $O(3, q)$, then $\pi$ contains secants and
  external lines. As $f_{BD}(K) = 1$ for secants $K$, and $f_{BD}(K) = 0$ for external lines $K$,
  this is a contradiction. 
  
  If $\pi \cap O^-(4, q)$ is isomorphic to $p^\perp$ for some point $p$ of $O^-(4, q)$,
  then $f_{BD}(K) = 1$ for $(q+1)/2$ tangents through $p$ and $f_{BD}(K) = 0$
  for the other $(q+1)/2$ tangents through $p$. If $p \notin L$, then again this is
  a contradiction, so $p \in L$. But we chose $L$ to be an external line, so $p \notin L$.
  
  Hence, $\pi^\pm \nrightarrow f_{BD'}$ for all hyperplanes $\pi$.
  Hence, $f_{BD'}^\pm = \bigvee p_i^+$ for $p_i$ on a plane $g$ which contains $L$.
  
  If $p_i \not\in \cP$, then $p_i$ lies on secants and passants not in $L$.
  As $f_{BD}(K) = 1$ if $K$ a secant and $f_{BD}(K) = 0$ if $K$ a passant,
  this is a contradiction.
  
  Hence, $p_i \in \cP$. Then $f_{BD}(K) = 1$ for all lines $K$ through $p_i$
  except for $(q+1)/2$ tangents in the plane $p_i^\perp$ for which we have $f_{BD}(K) = 0$. 
  Hence, $L \subseteq p_i^\perp$. But $L^\perp$ contains at most two points of $\cP$.
  Therefore, as $|p^+| = q^2$, either $|f_{BD'}| \leq 2 \cdot q^2$ or $|f_{BD'}| \geq q^4 - 2\cdot q^2$.
  By Lemma \ref{lem:BD_on_bilinear}, $|f_{BD'}| = (q^2+1)q^2/2$. As for $q \geq 3$,
  $2q^2 < (q^2+1)q^2/2 < q^4 - 2q^2$, this is a contradiction.
\end{proof}

\paragraph*{Remark} The published journal version of this document states Conjecture \ref{conj:bilinear_main} incorrectly.
Namely, it requires $g$ to be a line and not, as here, an $(\ell+1)$-space. As a consequence, the proof of 
Theorem \ref{lem:bilinear_sizes_canon} needed a small adjustment.

\section{Graphs from groups} \label{ch:groups}

\paragraph{Finite Abelian groups} Let $G = \prod_{i=1}^n \ZZ_{m_i}$ be a finite Abelian group. We think of functions on $G$ as $n$-variate functions whose $i$th input is $x_i \in \ZZ_{m_i}$. We say that a function on~$G$ has degree~1 if it has the form $f(x_1,\ldots,x_n) = \sum_{i=1}^n \phi_i(x_i)$, where $\phi_i\colon \ZZ_{m_i} \to \RR$ are arbitrary functions. When all $m_i$ are equal to some~$m$, then this agrees with the definition of degree~1 functions on the Hamming scheme $H(n,m)$.

\begin{theorem}[Folklore] \label{thm:abelian}
If $f$ is a Boolean degree~1 function on a finite Abelian group	$G = \prod_{i=1}^n \ZZ_{m_i}$ then $f(x_1,\ldots,x_n) = \phi_i(x_i)$ for some $i \in [n]$ and $\phi_i\colon \ZZ_{m_i} \to \{0,1\}$.
\end{theorem}
\begin{proof}
 Suppose that $f = \sum_{i=1}^n \phi_i(x_i)$. The claim clearly follows if we show that at most one $\phi_i$ can be non-constant. Suppose, for the sake of contradiction, that both $\phi_i$ and $\phi_j$ are non-constant, say $\phi_i(y_i) < \phi_i(z_i)$ and $\phi_j(y_j) < \phi_j(z_j)$. Restrict $f$ to a Boolean function $g$ on $\ZZ_{m_i} \times \ZZ_{m_j}$ by fixing the other coordinates arbitrarily. We reach a contradiction by considering the following chain of inequalities:
\[
 g(y_i,y_j) < g(z_i,y_j) < g(z_i,z_j).
 \qedhere
\]
\end{proof}

The same argument is clearly valid for all product domains.

\paragraph{Symmetric group} We can think of the symmetric group $S_n$ as the collection of all $n\times n$ permutation matrices. The degree of a function on $S_n$ is then the smallest degree of a polynomial in the entries which represents the function. Ellis, Friedgut and Pilpel~\cite{EFP} have determined all Boolean degree~1 functions on $S_n$:

\begin{theorem}[Ellis, Friedgut and Pilpel] \label{thm:efp}
 If $f$ is a Boolean degree~1 function on $S_n$ then either $f(\pi) = 1_{\pi(i) \in J}$ for some $i \in [n]$ and $J \subseteq [n]$, or $f(\pi) = 1_{\pi^{-1}(j) \in I}$ for some $j \in [n]$ and $I \subseteq [n]$.
\end{theorem}

Ellis, Friedgut and Pilpel also claim to characterize Boolean degree~$d$ functions on the symmetric group, but there is a mistake in their argument~\cite{efp-comment}.

\paragraph{General linear group} We can think of the general linear group $GL_q(n)$ as the collection of all $\frac{q^n-1}{q-1} \times \frac{q^n-1}{q-1}$ matrices with entries in $\FF_q$ which represent linear operators, and define degree accordingly. The following is a (possibly incomplete) list of Boolean degree~1 functions on $GL_2(n)$, stated as conditions on the input matrix $M \in GL_2(n)$:
\begin{itemize}
\item $Mx \in Y$.
\item $zM \in W$.
\item $Mx \notin Y$ and $wM^{-1} \notin Y^\perp$, where $w \not\perp x$.
\item $zM \notin W$ and $M^{-1}y \notin W^\perp$, where $z \not\perp y$.
\end{itemize}
Such a list can be used to derive our characterization of Boolean degree~1 functions on the Grassmann scheme in the case $q=2$, since the latter can be realized as the set of left cosets of a parabolic subgroup of $GL_2(n)$.

\section{Multislice graphs}\label{ch:multislice}

The multislices decompose $\ZZ_m^n$ (or, $H(n,m)$) in the same way that the slices decompose $\ZZ_2^n$ (or, $H(n,2)$). Let $k_1,\ldots,k_m$ be positive integers summing to $n$. The \emph{multislice} $M(k_1,\ldots,k_m)$ consists of all points of $\ZZ_m^n$ in which the number of coordinates colored~$i$ is equal to $k_i$. We can think of the elements as vectors in $\Omega_m^n$ (where $\Omega_m$ consists of all complex $m$th roots of unity) with given histogram, and then the degree of a function is the minimal $d$ such that the function is a linear combination of monomials which involve at most $d$ different coordinates. Alternatively, we can think of $\Omega_m$ as an abstract set consisting of $m$ ``colors''.

An alternative encoding uses a two-dimensional $0,1$~array $x_{ij}$ (where $1 \leq i \leq n$ and $1 \leq j \leq m$) in which the rows sum to~1 and the columns sum to~$k_i$. (If we do not put any restriction on the columns, we get all of $\ZZ_m^n$, which in this context we call the \emph{multicube} $H(n, m)$.) When $m=n$ and $k_1=\cdots=k_n=1$, the multislice becomes the symmetric group, and the array becomes a permutation matrix. More generally, a multislice is just a permutation module of the symmetric group. We define the degree of a function as the minimal degree of a polynomial representing the function. It is not hard to check that the two definitions are equivalent.

We show an analog of Lemma~\ref{thm:johnson-deg1_intro}.

\begin{lemma} \label{lem:multislice-deg1}
 Let $k_1,\ldots,k_m \geq 1$, and let $D = \{ i : k_i = 1 \}$. A Boolean degree~1 function on $M(k_1,\ldots,k_m)$ either depends on the color of some coordinate, or on which coordinate gets color $c$ for some $c \in D$.
\end{lemma}
\begin{proof}
 Let $f$ be a $0,1$-valued degree~1 function on $M(k_1,\ldots,k_m)$. We lift $f$ to a Boolean function $F$ on $S_n$ as follows. Let $\chi\colon [n] \to [m]$ map the first $k_1$~values the color~$1$, the following $k_2$~values the color~$2$, and so on. We define $F(\pi) = f(\chi(\pi(1)),\ldots,\chi(\pi(n)))$. 
 We claim that $\deg F = 1$. Indeed, denoting the input to $f$ by $x_{ij}$ (using the two-dimensional array input convention) and the input to $F$ by $X_{ij}$, we see that $x_{ij}$ is a sum of $k_j$ values of the form $X_{iJ}$. Theorem~\ref{thm:efp} thus implies that one of the following cases holds:
\begin{enumerate}
 \item There exists $i \in [n]$ and $J \subseteq [n]$ such that $F(\pi)$ is the indicator of ``$\pi(i) \in J$''.
 \item There exists $j \in [n]$ and $I \subseteq [n]$ such that $F(\pi)$ is the indicator of ``$\pi^{-1}(j) \in I$''.
\end{enumerate}
 Let us say that two permutations $\sigma,\tau$ are equivalent if $\chi(\sigma(i)) = \chi(\tau(i))$ for all~$i$. The definition of $F$ guarantees that $F(\sigma) = F(\tau)$ if $\sigma,\tau$ are equivalent. Therefore the condition indicated by $F$ must be invariant under equivalence. Let us now consider the two cases above.
 
 \textbf{Case 1:} $F(\pi)$ is the indicator of ``$\pi(i) \in J$''. In this case $f$ depends only on $x_i$.
 
 \textbf{Case 2:} $F(\pi)$ is the indicator of ``$\pi^{-1}(j) \in I$''. If $\chi(j') = \chi(j)$ then this condition must be equivalent to $\pi^{-1}(j') \in I$. If $j' \neq j$ and $I \notin \{\emptyset,[n]\}$ then we can find a permutation $\pi$ such that $\pi^{-1}(j) \in I$ and $\pi^{-1}(j') \notin I$, and so we reach a contradiction. This shows that either $\chi(j) \in D$, in which case $f$ depends on which coordinate gets color $\chi(j)$, or $I \in \{\emptyset,[n]\}$, in which case $f$ is constant.
\end{proof}


\section{Future work} \label{sec:future}

In order to complete the classification of Boolean degree $1$ functions on Grassmann graphs, it is sufficient to classify Boolean degree $1$ functions on $J_q(n,k)$ for $k$ very small, ideally $k=2$.
Hence, proving Theorem~\ref{thm:k_eq_2_results} for $n \geq 5$ without relying on the classification
of Boolean degree $1$ function on $J_q(4, 2)$ would be ideal. Perhaps results such as the one
given in \cite{Gavrilyuk2014,Metsch2014} could help here.

For polar spaces, it would be very interesting to investigate other
small cases by computer as was done for $J_q(4, 2)$, so that the validity of Conjecture~\ref{conj:polar_main}
can be further validated. In particular, for $k, n-k \geq 2$ and $q \in \{ 2, 3, 4, 5 \}$, this would be enough
to extend the classification result given in Theorem~\ref{thm:polar_small_dim_main} due to Theorem~\ref{thm:polar_inductive_step}.

Similar questions as for the Grassmann graph and polar spaces arise for the sesquilinear forms graphs.

\smallskip

Often Boolean degree $1$ functions correspond to the largest families of intersecting objects,
connecting them to Erd\H{o}s--Ko--Rado (EKR) theorems (see \cite{Godsil2016}). Indeed, all our trivial
examples are built from these intersecting families, which are the is indicator functions $x_i^+$.
In the group case, EKR theorems are known for all $2$-transitive groups, and the largest examples are indeed Boolean degree $1$ functions \cite{Meagher2016},
so classifying all Boolean degree $1$ functions on $2$-transitive groups would be very interesting.

\smallskip

One of the classical results in analysis of Boolean functions on the hypercube, the Friedgut--Kalai--Naor theorem~\cite{FKN}, states that a Boolean function on the hypercube which almost has degree~1 (in the sense that it is close in $L_2$ norm to a degree~1 function, which is not necessarily Boolean) is close to a Boolean degree $1$ function. This has implications to EKR theory, since an almost largest family of intersecting objects is often close to degree~$1$. Similar results have been proven on the Hamming graphs~\cite{ADFS}, Johnson graphs~\cite{filmus5}, and symmetric groups~\cite{eff1,eff2}, but not on any of the other domains considered here. Generalizations to larger degree have also been considered on the hypercube~\cite{Kindler,KindlerSafra}, Johnson graphs~\cite{KellerKlein}, and symmetric groups~\cite{eff3}. 

\smallskip

In a different direction, it would be interesting to extend the classification results to Boolean degree~$d$ functions. For a domain $D$, let $\gamma_d(D)$ be the largest number of coordinates that a Boolean degree~$d$ function on $D$ can depend on. Classical results in analysis of Boolean functions on the hypercube show that $\Omega(2^d) \leq \gamma_d(H(n,2)) \leq d2^{d-1}$. In ongoing work~\cite{FilmusIhringer2}, we have shown that for $k,n-k \geq \exp(d)$ it holds that $\gamma_d(J(n,k)) = \gamma_d(H(n,2))$, and we suspect that a similar result holds for the multislice. More generally, we conjecture that in the domains considered in the paper, Boolean degree~$d$ functions can be formed by combining a bounded number of Boolean degree~$1$ functions.

\paragraph*{Acknowledgements} We thank Alexander L. Gavrilyuk and Maarten De Boeck for their various remarks on earlier drafts of this document.


\begin{thebibliography}{10}

\bibitem{ADFS}
Noga Alon, Irit Dinur, Ehud Friedgut, and Benny Sudakov.
\newblock Graph products, {F}ourier analysis and spectral techniques.
\newblock {\em Geometric and functional analysis}, 14(5):913--940, 2004.

\bibitem{Bamberg2007}
John Bamberg, Shane Kelly, Maska Law, and Tim Penttila.
\newblock Tight sets and {$m$}-ovoids of finite polar spaces.
\newblock {\em J. Combin. Theory Ser. A}, 114(7):1293--1314, 2007.

\bibitem{Brouwer2017}
Andries~E. Brouwer, Sebastian~M. Cioab\u{a}, Ferdinand Ihringer, and Matt
  McGinnis.
\newblock The smallest eigenvalues of {H}amming graphs, {J}ohnson graphs and
  other distance-regular graphs with classical parameters.
\newblock {\em arXiv:1709.09011 [math.CO]}, 2017.

\bibitem{Brouwer1989}
Andries~E. Brouwer, Arjeh~M. Cohen, and Arnold Neumaier.
\newblock {\em Distance-regular graphs}, volume~18 of {\em Ergebnisse der
  Mathematik und ihrer Grenzgebiete (3) [Results in Mathematics and Related
  Areas (3)]}.
\newblock Springer-Verlag, Berlin, 1989.

\bibitem{Brouwer2003}
Andries~E. Brouwer, Chris~D. Godsil, Jack~H. Koolen, and William~J. Martin.
\newblock Width and dual width of subsets in polynomial association schemes.
\newblock {\em J. Combin. Theory Ser. A}, 102(2):255--271, 2003.

\bibitem{Bruen1999}
Aiden~A. Bruen and Keldon Drudge.
\newblock The construction of {C}ameron-{L}iebler line classes in {${\rm
  PG}(3,q)$}.
\newblock {\em Finite Fields Appl.}, 5(1):35--45, 1999.

\bibitem{Cameron1982}
Peter~J. Cameron and Robert~A. Liebler.
\newblock Tactical decompositions and orbits of projective groups.
\newblock {\em Linear Algebra Appl.}, 46:91--102, 1982.

\bibitem{Cossidente2017}
Antonio Cossidente and Francesco Pavese.
\newblock New {C}ameron--{L}iebler line classes with parameter
  {$\frac{q^2+1}{2}$}.
\newblock {\em arXiv:1707.01878 [math.CO]}, 2017.

\bibitem{DeBeule2016}
Jan De~Beule, Jeroen Demeyer, Klaus Metsch, and Morgan Rodgers.
\newblock A new family of tight sets in {$Q^+(5,q)$}.
\newblock {\em Des. Codes Cryptogr.}, 78(3):655--678, 2016.

\bibitem{DeBoeck2017}
Maarten {De Boeck}, Morgan Rodgers, Leo Storme, and Andrea Svob.
\newblock {C}ameron--{L}iebler sets of generators in finite classical polar
  spaces.
\newblock {\em arXiv:1712.06176 [math.CO]}, 2017.

\bibitem{DeBoeck2016}
Maarten De~Boeck, Leo Storme, and Andrea Svob.
\newblock The {C}ameron--{L}iebler problem for sets.
\newblock {\em Discrete Math.}, 339(2):470--474, 2016.

\bibitem{DeBruyn2010}
Bart De~Bruyn and Hiroshi Suzuki.
\newblock Intriguing sets of vertices of regular graphs.
\newblock {\em Graphs Combin.}, 26(5):629--646, 2010.

\bibitem{Delsarte1973}
Philippe Delsarte.
\newblock An algebraic approach to the association schemes of coding theory.
\newblock {\em Philips Res. Rep. Suppl.}, (10):vi+97, 1973.

\bibitem{Delsarte1978}
Philippe Delsarte.
\newblock Bilinear forms over a finite field with applications to coding
  theory.
\newblock {\em J. Combin. Th. (A)}, (25):226--241, 1978.

\bibitem{TR16-198}
Irit Dinur, Subhash Khot, Guy Kindler, Dor Minzer, and Muli Safra.
\newblock Towards a proof of the 2-to-1 games conjecture?
\newblock Technical Report TR16-198, ECCC, 2016.

\bibitem{TR17-094}
Irit Dinur, Subhash Khot, Guy Kindler, Dor Minzer, and Muli Safra.
\newblock On non-optimally expanding sets in {G}rassmann graphs.
\newblock Technical Report TR17-094, ECCC, 2017.

\bibitem{Drudge1998}
Keldon~W. Drudge.
\newblock {\em Extremal sets in projective and polar spaces}.
\newblock PhD thesis, The University of Western Ontario, 1998.

\bibitem{Eisfeld1998}
J{\"o}rg Eisfeld.
\newblock Subsets of association schemes corresponding to eigenvectors of the
  {B}ose-{M}esner algebra.
\newblock {\em Bull. Belg. Math. Soc. Simon Stevin}, 5(2-3):265--274, 1998.
\newblock Finite geometry and combinatorics (Deinze, 1997).

\bibitem{Eisfeld1999}
J{\"o}rg Eisfeld.
\newblock The eigenspaces of the {B}ose-{M}esner algebras of the association
  schemes corresponding to projective spaces and polar spaces.
\newblock {\em Des. Codes Cryptogr.}, 17(1-3):129--150, 1999.

\bibitem{eff1}
David Ellis, Yuval Filmus, and Ehud Friedgut.
\newblock A quasi-stability result for dictatorships in {$S_n$}.
\newblock {\em Combinatorica}, 35(5):573--618, 2015.

\bibitem{eff2}
David Ellis, Yuval Filmus, and Ehud Friedgut.
\newblock A stability result for balanced dictatorships in {$S_n$}.
\newblock {\em Random Structures and Algorithms}, 46(3):494--530, 2015.

\bibitem{eff3}
David Ellis, Yuval Filmus, and Ehud Friedgut.
\newblock Low-degree {B}oolean functions on {$S_n$}, with an application to
  isoperimetry.
\newblock {\em Forum of Mathematics, Sigma}, 5, 2017.

\bibitem{EFP}
David Ellis, Ehud Friedgut, and Haran Pilpel.
\newblock Intersecting families of permutations.
\newblock {\em J. Amer. Math. Soc.}, 24:649--682, 2011.

\bibitem{Feng2015}
Tao Feng, Koji Momihara, and Qing Xiang.
\newblock Cameron-{L}iebler line classes with parameter {$x=\frac{q^2-1}{2}$}.
\newblock {\em J. Combin. Theory Ser. A}, 133:307--338, 2015.

\bibitem{filmus5}
Yuval Filmus.
\newblock Friedgut--{K}alai--{N}aor theorem for slices of the {B}oolean cube.
\newblock {\em Chicago Journal of Theoretical Computer Science}, pages
  14:1--14:17, 2016.

\bibitem{filmus4}
Yuval Filmus.
\newblock Orthogonal basis for functions over a slice of the {B}oolean
  hypercube.
\newblock {\em Electronic J. Comb.}, 23(1):P1.23, 2016.

\bibitem{efp-comment}
Yuval Filmus.
\newblock {A comment on Intersecting Families of Permutations}.
\newblock {\em CoRR}, abs/1706.10146, 2017.

\bibitem{FilmusIhringer2}
Yuval Filmus and Ferdinand Ihringer.
\newblock Boolean constant degree functions on the slice are juntas.
\newblock {\em CoRR}, abs/1801.06338, 2018.

\bibitem{fkmw}
Yuval Filmus, Guy Kindler, Elchanan Mossel, and Karl Wimmer.
\newblock Invariance principle on the slice.
\newblock In {\em 31st Computational Complexity Conference}, 2016.

\bibitem{fm}
Yuval Filmus and Elchanan Mossel.
\newblock Harmonicity and invariance on slices of the {B}oolean cube.
\newblock {\em Probability Theory and Related Fields}, to appear.

\bibitem{FKN}
Ehud Friedgut, Gil Kalai, and Assaf Naor.
\newblock Boolean functions whose {F}ourier transform is concentrated on the
  first two levels.
\newblock {\em Advances in Applied Mathematics}, 29(3):427--437, 2002.

\bibitem{Gavrilyuk2018a}
Alexander~L. Gavrilyuk and I.~Matkin.
\newblock Cameron-liebler line classes in ${PG}(3,5)$.
\newblock {\em arXiv:1803.10442 [math.CO]}, 2018.

\bibitem{Gavrilyuk2018}
Alexander~L. Gavrilyuk, Ilia Matkin, and Tim Penttila.
\newblock Derivation of {C}ameron-{L}iebler line classes.
\newblock {\em Des. Codes Cryptogr.}, 86:231--236, 2018.

\bibitem{Gavrilyuk2014}
Alexander~L. Gavrilyuk and Klaus Metsch.
\newblock A modular equality for {C}ameron-{L}iebler line classes.
\newblock {\em J. Combin. Theory Ser. A}, 127:224--242, 2014.

\bibitem{Gavrilyuk2014a}
Alexander~L. Gavrilyuk and Ivan~Yu. Mogilnykh.
\newblock Cameron-{L}iebler line classes in {$PG(n,4)$}.
\newblock {\em Des. Codes Cryptogr.}, 73(3):969--982, 2014.

\bibitem{Godsil2016}
Chris Godsil and Karen Meagher.
\newblock {\em {E}rd{\H o}s-{K}o-{R}ado Theorems: Algebraic Approaches}.
\newblock Number 149 in Cambridge Studies in Advanced Mathematics. Cambridge
  Univ. Press, December 2016.

\bibitem{Govaerts2005}
Patrick Govaerts and Tim Penttila.
\newblock Cameron-{L}iebler line classes in {${\rm PG}(3,4)$}.
\newblock {\em Bull. Belg. Math. Soc. Simon Stevin}, 12(5):793--804, 2005.

\bibitem{Hua1949}
Loo-Keng Hua.
\newblock Geometry of symmetric matrices over any field with characteristic
  other than two.
\newblock {\em Ann. of Math. (2)}, 50:8--31, 1949.

\bibitem{KellerKlein}
Nathan Keller and Ohad Klein.
\newblock Kindler--{S}afra theorem for the slice.
\newblock Manuscript, 2017.

\bibitem{Khot:2017:ISG:3055399.3055432}
Subhash Khot, Dor Minzer, and Muli Safra.
\newblock On independent sets, 2-to-2 games, and grassmann graphs.
\newblock In {\em Proceedings of the 49th Annual ACM SIGACT Symposium on Theory
  of Computing}, STOC 2017, pages 576--589, New York, NY, USA, 2017. ACM.

\bibitem{TR18-006}
Subhash Khot, Dor Minzer, and Muli Safra.
\newblock Pseudorandom sets in {G}rassmann graph have near-perfect expansion.
\newblock Technical Report TR18-006, ECCC, 2018.

\bibitem{Kindler}
Guy Kindler.
\newblock {\em Property testing, {PCP} and Juntas}.
\newblock PhD thesis, Tel-Aviv University, 2002.

\bibitem{KindlerSafra}
Guy Kindler and Shmuel Safra.
\newblock Noise-resistant {B}oolean functions are juntas.
\newblock {\em Manuscript}.

\bibitem{Martin1994}
William~J. Martin.
\newblock Completely regular designs of strength one.
\newblock {\em J. Algebraic Combin.}, 3(2):177--185, 1994.

\bibitem{Matkin2018}
I.~Matkin.
\newblock Cameron-liebler line classes in ${PG}(n, 5)$.
\newblock {\em Trudy Inst. Mat. i Mekh. UrO RAN}, 24(2):158--172, 2018.

\bibitem{Meagher2016}
Karen Meagher, Pablo Spiga, and Pham~Huu Tiep.
\newblock An {E}rd{\H o}s--{K}o--{R}ado theorem for finite 2-transitive groups.
\newblock {\em European J. Combin.}, 550:100--118, 2016.

\bibitem{Metsch2014}
Klaus Metsch.
\newblock An improved bound on the existence of {C}ameron-{L}iebler line
  classes.
\newblock {\em J. Combin. Theory Ser. A}, 121:89--93, 2014.

\bibitem{Metsch2017}
Klaus Metsch.
\newblock A gap result for {C}ameron-{L}iebler {$k$}-classes.
\newblock {\em Discrete Math.}, 340(6):1311--1318, 2017.

\bibitem{Meyerowitz1992}
Aaron~D. Meyerowitz.
\newblock Cycle-balanced partitions in distance-regular graphs.
\newblock {\em J. Combin. Inform. System Sci.}, 17(1-2):39--42, 1992.

\bibitem{RyanODonnell}
Ryan O'Donnell.
\newblock {\em Analysis of {B}oolean functions}.
\newblock Cambridge University Press, 2014.

\bibitem{doi:10.1137/100787325}
Ryan O'Donnell and Karl Wimmer.
\newblock {KKL}, {K}ruskal--{K}atona, and monotone nets.
\newblock {\em SIAM Journal on Computing}, 42(6):2375--2399, 2013.

\bibitem{Pentilla1985}
Tim Pentilla.
\newblock {\em Collineations and Configurations in Projective Spaces}.
\newblock PhD thesis, Oxford University, 1985.

\bibitem{Penttila1991}
Tim Penttila.
\newblock Cameron-{L}iebler line classes in {${\rm PG}(3,q)$}.
\newblock {\em Geom. Dedicata}, 37(3):245--252, 1991.

\bibitem{Rodgers2018}
Morgan Rodgers, Leo Storme, and Andries Vansweevelt.
\newblock The {C}ameron-{L}iebler $k$-classes in $pg(2k+1, q)$.
\newblock {\em Combinatorica}, in press.

\bibitem{Schmidt2018}
Kai-Uwe Schmidt.
\newblock Hermitian rank distance codes.
\newblock {\em Des. Codes Cryptogr.}, 2018.

\bibitem{Stanton1980a}
Dennis Stanton.
\newblock Some {$q$}-{K}rawtchouk polynomials on {C}hevalley groups.
\newblock {\em Amer. J. Math.}, 102(4):625--662, 1980.

\bibitem{Terwilliger}
Paul Terwilliger.
\newblock Quantum matroids.
\newblock In E.~Bannai and A.~Munemasa, editors, {\em Progress in Algebraic
  Combinatorics}, volume~24 of {\em Advanced Studies in Pure Mathematics},
  pages 323--441, Tokyo, 1996. Mathematical Society of Japan.

\bibitem{Tits1974}
Jacques Tits.
\newblock {\em Buildings of spherical type and finite {BN}-pairs}.
\newblock Lecture Notes in Mathematics, Vol. 386. Springer-Verlag, Berlin,
  1974.

\bibitem{Vanhove2011}
Fr{\'e}d{\'e}ric Vanhove.
\newblock {\em Incidence geometry from an algebraic graph theory point of
  view}.
\newblock PhD thesis, Ghent University, 2011.

\end{thebibliography}

\end{document}